\documentclass{amsart}
\usepackage{amssymb}
\usepackage{amsmath}
\usepackage{graphicx}
\usepackage{color}
\usepackage{amsfonts}

\newcommand{\e}{\varepsilon}
\newcommand{\ue}{u^{\varepsilon}}
\newcommand{\ueeta}{u^{\varepsilon,\eta}}
\newcommand{\ba}{{\bf a}}
\newcommand{\bx}{{\bf x}}
\newcommand{\bs}{{\bf s}}
\newcommand{\br}{{\bf r}}
\newcommand{\bF}{{\bf F}}
\newcommand{\by}{{\bf y}}
\newcommand{\bzeta}{{\bf{\zeta}}}
\usepackage{enumerate}

\newtheorem{Remark}{Remark}
\newtheorem{Theorem}{Theorem}
\newtheorem{Corollary}{Corollary}
\newtheorem{Example}{Example}
\newtheorem{Definition}{Definition}
\newtheorem{Lemma}{Lemma}
\definecolor{orange}{rgb}{0,0,0}
\definecolor{blue}{rgb}{0,0,0}
\definecolor{red}{rgb}{0,0,0}

\begin{document}

\title[Upscaling error in locally periodic media]{Estimates for the upscaling error in heterogeneous multiscale methods for wave propagation problems in locally periodic media}

\author[D.~Arjmand]{Doghonay Arjmand}
\author[O.~Runborg]{Olof Runborg}
\email{doghonay.arjmand@it.uu.se,olofr@kth.se}

\address{Division of Scientific Computing \\
  Department of Information Technology \\
  Uppsala University \\
  SE-751 05 Uppsala, Sweden.}
 
\address{Division of numerical analysis \\
  Department of Mathematics \\
  The Royal Institute of Technology \\
  Stockholm, Sweden.}

\begin{abstract}
This paper concerns the analysis of a multiscale method for wave propagation problems in microscopically nonhomogeneous media. A direct numerical approximation of such problems is prohibitively expensive as it requires resolving the microscopic variations over a much larger physical domain of interest. The heterogeneous multiscale method (HMM) is an efficient framework to approximate the solutions of multiscale problems. In HMM, one assumes an incomplete macroscopic model which is coupled to a known but expensive microscopic model. The micromodel is solved only locally to upscale the parameter values which are missing in the macromodel. The resulting macroscopic model can then be solved at a cost independent of the small scales in the problem.

In general, the accuracy of the HMM is related to how good the upscaling step approximates the right macroscopic quantities. The analysis of the method, that we consider here, was previously addressed only in purely periodic media although the method itself is numerically shown to be applicable to more general settings. In the present study, we consider a more realistic setting by assuming a locally-periodic medium where slow and fast variations are allowed at the same time. We then prove that HMM captures the right macroscopic effects. The generality of the tools and ideas in the analysis allows us to establish convergence rates in a multi-dimensional setting. The theoretical findings here imply an improved convergence rate in one-dimension, which also justifies the numerical observations from our earlier study.
\end{abstract}

\maketitle

\section{Introduction}
We consider the scalar wave equation in locally-periodic media
\begin{align} \label{eqn_Main}
\partial_{tt} \ue(t,\bx) &= \nabla \cdot   \Big( A(\bx,\bx/\e) \nabla \ue(t,\bx) \Big) + f(t,\bx), \quad \text{ in } \Omega \times (0,T]  \\
\ue(0,\bx) &= g(\bx), \quad \partial_t \ue(0,\bx) = h(\bx), \quad \text{ on } \Omega \times \{ t=0 \},   \nonumber \\
\ue(t,\bx) &= 0 \quad \text{ on } \partial \Omega \times [0,T] , \nonumber
\end{align} 
where $\Omega$ is a bounded open subset of $\mathbb{R}^d$ with $|\Omega| = O(1)$, and $A$ is a bounded symmetric positive-definite matrix function in $\mathbb{R}^{d \times d}$ such that for every $\bzeta \in \mathbb{R}^{d}$
\begin{equation*}
c_1 |\bzeta|^{2} \leq \sup_{\bx \in \Omega,\by \in Y} \zeta^{T} A(\bx,\by) \bzeta \leq c_2 \left| \bzeta \right|^{2}, \quad Y = (0,1]^{d}, \text{ and } A_{ij} = A_{ji}.
\end{equation*}
The coefficient $A$ is assumed to be locally-periodic, i.e. $\by \longrightarrow A(\bx,\by)$ is $Y$-periodic for all $\bx$ in $\Omega$, the parameter $\e \ll 1$ represents the wavelength of the small scale variations in the media, and $T=O(1)$ is a constant independent of $\e$. For simplicity, we will assume that $A$ is smooth, i.e., $A_{ij} \in C^{\infty}(\overline{\Omega} \times \overline{Y})$, but most of the theoretical results in this paper will be valid for less regular coefficients such as $\partial_{\bx}^{k} A_{ij} \in C(\overline{\Omega},L^{\infty}(Y))$ for $k\leq 2$.

When $\e \ll 1$, the solution to equation \eqref{eqn_Main} exhibits variations at a coarse and a fine scale, in time and space, due to the heterogeneities in the coefficient function $A$. In this case, a direct numerical simulation (DNS) of  \eqref{eqn_Main} becomes highly demanding since $\e$-scale variations must be resolved over the entire domain, leading to $O(\e^{-d-1})$ degrees of freedom. Although DNS provides us with detailed information about the behavior of the solution at different scales, in engineering practices, we are often interested only in a coarse scale description of the solution. This raises the following question: can we approximate the coarse scale, i.e. a local average, part of the solution at a cost much lower than the cost of DNS? From a mathematical point of view this is related to the theory of homogenization which can be traced back to $1970$s. On the other hand, from a numerical point of view, this issue has triggered the birth and development of a number of successful numerical multiscale approaches over the last two decades.

First we give a brief summary of the idea behind analytical homogenization. The term homogenization was introduced to the mathematical  literature by I. Babuska, \cite{Babuska} and since then the theory has been systematically developed by contributions of various researchers, see e.g. \cite{Bensoussan_Lions_Papanicolaou,Cioranescu_Donato,Jikov_Kozlov_Oleinik,Pavliotis_Stuart} for an exposition of the method, where  references to earlier literature can also be found. Mathematically speaking, the aim of the homogenization theory is to find a limiting solution $u^{\e} \longrightarrow u^{0}$, in some appropriate sense, and possibly a homogenized problem, satisfied by $u^0$, which is no more dependent on the small scale parameter $\e$. In a few cases, it is possible to write down explicit equations for the homogenized problem. For example, when the medium is periodic or locally periodic, the homogenized solution $u^{0}$ solves 
\begin{align}  \label{Eqn_Intro_Homogenization}
\partial_{tt} u^0(t,\bx) &= \nabla \cdot   \Big( A^{0}(\bx) \nabla u^0(t,\bx) \Big) + f(t,\bx), \quad \text{ in } \Omega \times (0,T]  \\
u^0(0,\bx) &= g(\bx), \quad \partial_t u^0(0,\bx) = h(\bx), \quad \text{ on }  \Omega \times \{ t=0 \},   \nonumber
\end{align} 
where the homogenized matrix $A^{0}$ is given by
\begin{equation}\label{eqn_C_HatA_Def} 
A^{0}_{ij}(\bx) = \int_Y \left( A_{ij}(\bx,\by)  + \sum_{k=1}^{d} A_{ik} \partial_{y_k} \chi_{j}(\bx,\by) \right) \; d\by,
\end{equation}
and $\{ \chi_{\ell} \}_{\ell=1}^{d}$ are $Y$-periodic solutions of the following set of cell problems
\begin{align} \label{eqn_CellProblemChi}
\nabla_{\by} \cdot \left(  A(\bx,\by) \nabla_{\by} \chi_{\ell}(\bx,\by)    + A(\bx,\by) e_{\ell}  \right) = 0, \quad \int_{Y} \chi_{\ell}(\bx,\by) \; d\by = 0,
\end{align}
where $\{  e_{\ell} \}_{\ell=1}^{d}$ are canonical basis vectors in $\mathbb{R}^{d}$. Although the homogenization theory provides us with powerful analytical tools to study existence and uniqueness of homogenized solutions, it is of limited practical use since an explicit representation for the homogenized coefficient $A^{0}$ is often missing. In such a case, the tendency is to design multiscale numerical methods which target the coarse scale behavior of the solution, i.e., the homogenized solution, without assuming a priori knowledge about the homogenized coefficient or the precise nature of the coefficient $A$. Within the last two decades, several multiscale strategies have been proposed to approximate the coarse-scale dynamics of problems which possess variations at multiple scales. To name a few, the variational multiscale method (VMM) pioneered by Hughes et al. \cite{Hughes_Feijoo_Mazzei_Quincy}, the multiscale finite element method (MsFEM) by T. Hou et al. \cite{Hou_Wu}, the equation-free approach by Kevrikidis et al. \cite{Kevrikidis_etal}, and the heterogeneous multiscale method (HMM) due to E and Engquist \cite{E_Engquist_1} are successful examples of such general frameworks. Due to the large literature available on these methodologies, at this stage, we refer the curious reader only to few representative works in the context of applications to multiscale PDEs, see e.g. \cite{Abdulle_Henning_1,Henning_Malqvist_1,Malqvist_Peterseim} for methods inspired by the VMM principle, \cite{Efendiev_Hou_Book,Efendiev_Hou_Wu,Hou_Wu_Zhang,Henning_Peterseim_1} for MsFEM type methods, and \cite{Arjmand_Stohrer,Engquist_Holst_Runborg_1,Abdulle_Grote_1} for HMM based strategies. See also \cite{Babuska_Lipton,Engquist_Runborg_1,Gloria_1,Gloria_2,Owhadi_Zhang_1,Owhadi_Zhang_Berlyand,Schwab_1} for other relevant literatures on multiscale approaches for PDEs.

The focus of the present work is  on the HMM strategy. The idea behind HMM is to assume a macroscale model which lacks certain input data. To close the macromodel, one then solves a micromodel locally and computes effective parameter values needed for the macromodel. Since the micromodel is solved only locally, the approach leads to a significant improvement in terms of computational cost in comparison to traditional numerical schemes. HMM has been applied to a wide range of multiscale problems, see e.g. \cite{Abdulle,Abdulle_Vilmart_1,Arjmand_Runborg_1,Arjmand_Stohrer,E_Ming_Zhang} for elliptic, parabolic and second-order hyperbolic PDEs, \cite{Ren_E} for applications in micro-fluidics, \cite{Engquist_Tsai_1,E} for applications in ODEs with multiple scales, and \cite{Abdulle_E_Engquist_1} for a recent overview of the method. 

In \cite{E_Engquist_1,Abdulle_E_Engquist_1}, a general framework for the analysis of HMM for multiscale PDEs is given. The idea is to split the error between the HMM and the homogenized solution \eqref{Eqn_Intro_Homogenization} into three parts
\begin{align*}
\left| U_{HMM} - u^{0} \right|  =  e_{macro} + e_{upscaling} + e_{micro},
\end{align*}
{\color{red}{ where $e_{macro}$ and $e_{micro}$ are discretization errors, and $e_{upscaling}$ or so-called the HMM error is related to the accuracy of \emph{upscaling} procedure where effective parameters in the macromodel are computed using local microscopic simulations, {\color{red} see Remark \ref{Remark_UpscalingError}}. The aim of this study is to estimate the difference between the upscaled effective parameters and the exact homogenized quantities. A fully discrete analysis of a finite element HMM for elliptic PDEs can be found e.g. in \cite{Abdulle}. The analysis of the discretization errors is omitted in the present work and can be carried out using standard theory of finite differences or finite elements, see e.g. \cite{Brenner_Scott,Gustafsson_Kreiss_Oliger,Samarskii_Book}.}}

In this paper, we analyze a finite difference HMM (FD-HMM) from \cite{Engquist_Holst_Runborg_1} which approximates the solution of the second order wave equation \eqref{eqn_Main}, see Section \ref{HMM_Sec} for a summary of the FD-HMM. The analysis of this FD-HMM in purely periodic media was previously addressed  in \cite{Engquist_Holst_Runborg_1} for short time problems where $T = O(1)$, and in \cite{Arjmand_Runborg_2,Engquist_Holst_Runborg_LongTime} for long time scales where $T = O(\e^{-2})$. To be able to go beyond the rather academic case of periodic media and to address a more realistic scenario, an extension of the theory to non-periodic media is needed. The difficulty lies in the fact that existing theoretical results in the periodic case do not directly apply to non-periodic media. {\color{blue} The analysis in this paper does not fully cover the general non-periodic theory, but is based on two main assumptions: i) as a special case of non-periodic coefficients, locally-periodic coefficients, see the coefficient $A$ in \eqref{eqn_Main}, are assumed. ii) A scale separation, $\e \ll 1$, in the coefficient $A$ (and hence in the solution) is assumed. }

This article is structured as follows. In Section \ref{HMM_Sec}, we introduce the multiscale method. A detailed analysis of the method is then given in Section \ref{Analysis_Sec}. We finish the article by a conclusion in Section \ref{Discussion_Sec}.

\section{HMM}
\label{HMM_Sec}
The FD-HMM uses the following macromodel for approximating the solution of problem \eqref{eqn_Main}
\begin{equation}
\label{Intro_Macro_Model}
\text{Macro problem: }
\begin{array}{lll}
 \partial_{tt} U(t,\bx)  - \nabla \cdot \bF(\bx,\nabla U) = f(\bx), \quad \text{ in } \Omega \times (0,T]   \\
U(0,\bx) = g(\bx), \quad \partial_t U(0,\bx) = h(\bx), \quad \text{ on } \Omega \times \{ t=0 \}, \\
U(t,\bx) = 0, \quad \text{ on } \partial\Omega \times [0,T].
\end{array}
\end{equation}

Here $U$ is the macroscopic solution and $\bF=(F^{1},F^{2}, \cdots, F^{d})$ is the missing data in the model. A finite difference discretization (in two dimensions) of the macro problem \eqref{Intro_Macro_Model} gives
\begin{equation} \label{Macro_Solver_Eqn}
\begin{array}{ll}
U_{i,j}^{n+1} = 2 U_{i,j}^{n} - U_{i,j}^{n-1} + \triangle t^{2} \left( \dfrac{F_{i+\frac{1}{2},j}^{1,n} - F_{i-\frac{1}{2},j}^{1,n} }{H} + \dfrac{F_{i,j+\frac{1}{2}}^{2,n} - F_{i,j-\frac{1}{2}}^{2,n}}{H} \right) + \triangle t^{2} f_{i,j}^{n}. 
\end{array}
\end{equation}
Moreover,  $U^0_{i,j} = g_{i,j}$, and $U^{1}_{i,j}$ is given by 
\begin{align*}
U^{1}_{i,j} & \approx U(\triangle t, \bx_{i,j}) \approx U(0,\bx_{i,j}) + \triangle t \partial_t U(0,\bx_{i,j}) + \dfrac{\triangle t^2}{2} \partial_{tt} U(0,\bx_{i,j}),
\end{align*}
where the initial data in \eqref{Intro_Macro_Model} can be used to compute the first two terms in the right hand side, and the last term is computed by using the equation \eqref{Intro_Macro_Model} which also requires computing $\bF$ at time $t=0$. To compute the unknown $\bF_{i+\frac{1}{2},j}^{n}$ in the macro solver \eqref{Macro_Solver_Eqn}, we solve the multiscale problem \eqref{eqn_Main} over a microscopic box $I_{\tau} \times \Omega_{\bx_{i+1/2,j}}$, where $I_{\tau} = (0,\tau/2]$ and $\tau/2$ is the final time for the microscopic simulations, and $\Omega_{\bx_{i+1/2,j},\eta}:= \bx_{i+1/2,j} + [-L_{\eta}, L_{\eta}]^{d}$ where $L_{\eta} \geq \frac{\eta}{2} + \frac{\tau}{2} \sqrt{|A|_{\infty}} $ and in practice $\tau =\eta = O(\e)$. In other words, we solve
\begin{equation}
\label{Intro_Micro_Problem}
\text{Micro problem: }
\begin{array}{lll}
\partial_{tt}u^{\e,\eta}(t,\bx)  - \nabla \cdot \left(A(\bx,\bx/\e) \nabla u^{\e,\eta} \right) = 0,  \text{ in } \Omega_{\bx_{i+1/2,j},\eta} \times I_{\tau}\\
u^{\e,\eta}(0,\bx)  = \hat{u}(\bx), \quad \partial_{t}u^{\e,\eta}(0,\bx)= 0,  \quad \text{ on } {\Omega_{\bx_{i+1/2,j},\eta}} \times \{ t=0 \}, \\
u^{\e,\eta}(0,\bx) - \hat{u}(\bx) \quad \text{ is periodic in } \Omega_{\bx_{i+1/2, j},\eta}. 
\end{array}
\end{equation}
where $\hat{u}(\bx)$ is a  linear approximation of the neighbouring coarse scale data. {\color{red} In one-dimension, it is given by $\hat{u}(x) = s (x-x_{i+1/2}) + (U_{i+1} + U_{i})/2$, where $s = (U_{i+1}-U_{i})/H$ is the slope at the point $x_{i+1/2}$, and in two dimensions (as well as higher dimensions) it is found by a linear least square approximation of $\{ U_{i,j+k},U_{i+1,j+k} \}_{k=-1}^{k=1}$. Hence, in general we have
\begin{equation*} \label{Eqn_MacroState}
\hat{u}(\bx) = \bs \cdot (\bx- \bx_{i+1/2,j}) + c_0,
\end{equation*}
where $\bs \in \mathbb{R}^{d}$ is the slope vector at the point $\bx_{i+1/2,j}$, and $c_0$ is a suitable constant.} {\color{red} Moreover, for the micro problem, other boundary conditions such as the Dirichlet condition $u^{\e,\eta}(\bx) = \hat{u}(\bx)$ can also be used.}
\begin{Remark} Note that if $\tau = \eta  = O(\e)$, the computational cost of solving the micro problem \eqref{Intro_Micro_Problem} becomes independent of $\e$ since the solution will contain only few oscillations, in time and space, within the microscopic domain.
\end{Remark}
For the local averaging we introduce the space $\mathbb{K}^{p,q}$ which consists of functions $K \in C^{q}(\mathbb{R})$ compactly supported in $[-1,1]$, and $K^{(q+1)} \in BV(\mathbb{R})$, where the derivative is understood in the weak sense and $BV$ is the space of functions with bounded variations on $\mathbb{R}$. Moreover, the parameter $p$ represents the number of vanishing moments
$$
\int_{\mathbb{R}} K(t) t^r dt  = 
\begin{cases}
1 & r=0, \\
0 & {\color{blue} 1\leq } r \leq p.
\end{cases}
$$
As local averaging takes place in a domain of size $\eta$, we consider the scaled kernel
$$
K_{\eta}(x) = \dfrac{1}{\eta} K(x/\eta).
$$
Finally, the flux $\bF_{i+1/2,j}$ is computed by 
\begin{equation}\label{Intro_HMM_Flux}
\bF_{i+1/2,j}  = \left( \mathcal{K}_{\tau,\eta} \ast A(\cdot,\cdot/\e) \nabla \ueeta(\cdot,\cdot) \right)(0,\bx_{i+1/2,j}),
\end{equation}
where
\begin{equation*}
\left( \mathcal{K}_{\tau,\eta} \ast f \right)(t,\bx)= \int_{t-\tau/2}^{t+\tau/2}\int_{\Omega_{\bx,\eta}} K_{\eta}(\tilde{\bx}-\bx) K_{\tau}(\tilde{t}- t) f(\tilde{t}, \tilde{\bx}) \; d\tilde{\bx} \; d\tilde{t},
\end{equation*}
and where in $d$-dimension, $K_{\eta}(\bx)$ is understood as 
$$
K_{\eta}(\bx) = K_{\eta}(x_1)K_{\eta}(x_2) \cdots K_{\eta}(x_d).
$$
{\color{red} Note that to compute the HMM solutions $U_{i,j}^{n}$, an approximation of the microscopic solution  $u^{\e,\eta}$ solving \eqref{Intro_Micro_Problem} as well as a numerical approximation for the integration \eqref{Intro_HMM_Flux} are needed. In \cite{Arjmand_Runborg_1} and \cite{Engquist_Holst_Runborg_1} a simple leap-frog scheme and a standard trapezoidal rule are used to approximate these quantities. }
\begin{Remark} Note that in the upscaling step \eqref{Intro_HMM_Flux}, we need  the values of the solution for the micro problem \eqref{Intro_Micro_Problem} in the time interval $[-\tau/2,0)$. This requires no additional cost since  the symmetry property $u^{\e,\eta}(t,\bx) = u^{\e,\eta}(-t,\bx)$ easily follows due to the condition $\partial_t u^{\e,\eta}(0,\bx) = 0$.
\end{Remark}
\begin{Remark} In general, larger values for the parameters $p$, and $q$ result in better approximation properties, see Lemma \ref{Lemma_Kernel}. Moreover, taking large $p,q$ does not increase the computational cost, see \cite{Holst_2011} for construction of such averaging functions for all $p,q$.
\end{Remark}
{\color{red} \begin{Remark}\label{Remark_UpscalingError} Let $u^{0}$ be the solution of the homogenized equation \eqref{Eqn_Intro_Homogenization}. Let $\bar{u}_{i,j}^{n} \approx u^{0}(\bx_{i,j},t_n)$ be the approximate solution which solves the numerical scheme \eqref{Macro_Solver_Eqn} but with $\bF_{i,j}^{n} := \bF((\nabla U)_{i,j}^{n})$ replaced by $\hat{\bF}((\nabla \bar{u})_{i,j}^{n})$, where $(\nabla \bar{u})_{i,j}^{n}$ represents the slope of the approximate solution $\bar{u}_{i,j}^{n}$ at the point $(t_n,\bx_{i,j})$, and $ \hat{\bF}((\nabla \bar{u})_{i,j}^{n}) := A^{0}(\bx_{i,j}) (\nabla \bar{u})_{i,j}^{n}$. Moreover, let $\tilde{U}_{i,j}^{n}$ be the HMM solution when the micro-problem \eqref{Intro_Micro_Problem} and the integral \eqref{Intro_HMM_Flux} is solved exactly. Then the difference between the HMM solution $U_{i,j}^{n}$ and the homogenized solution $u^0(t_n,\bx_{i,j})$ can be split as follows:
$$
\left|  U_{i,j}^{n}  - u^{0}(t_n,\bx_{i,j})  \right| \leq    \underbrace{\left|  U_{i,j}^{n}  - \tilde{U}_{i,j}^{n}  \right|}_{e_{micro}} + \underbrace{\left|  \tilde{U}_{i,j}^{n}   - \bar{u}_{i,j}^{n}   \right|}_{e_{upscaling}}  +  \underbrace{\left| \bar{u}_{i,j}^{n}    - u^{0}(t_n,\bx_{i,j})   \right| }_{e_{macro}}
$$
Let $D \cdot \bF_{i,j}^{n}$ be  the discrete approximation of the divergence operator $\nabla \cdot \bF(t,\bx)$ given in the macro-solver \eqref{Macro_Solver_Eqn}, and $D_t^2 U_{i,j}^{n} :=  (U_{i,j}^{n+1}- 2  U_{i,j}^{n} +  U_{i,j}^{n-1})/\triangle t^2$. Then the error $e_{i,j}^{n}:= \tilde{U}_{i,j}^{n}   - \bar{u}_{i,j}^{n}$ satisfies 
\begin{align*}
D^2_{t} e^{n}_{i,j}  &= D \cdot \hat{\bF}((\nabla e)_{i,j}^{n})  + D \cdot \left(   \hat{\bF}((\nabla \bar{u})_{i,j}^{n})   -   \bF((\nabla \bar{u})_{i,j}^{n})\right) \\
e^{0}_{i,j} &= 0, \quad e^{1}_{i,j} = \frac{\triangle t^2}{2} D \cdot \left(   \bF((\nabla \bar{u})_{i,j}^{0})   -   \hat{\bF}((\nabla \bar{u})_{i,j}^{0})\right).
\end{align*}
This shows that the upscaling error will be small if the difference between the fluxes $\bF$ and $\hat{\bF}$ is small in comparison to the macroscopic mesh size $H$, see \cite{Abulle_Weinan_2003} for a similar result in the parabolic setting. The aim of this paper is to show that the difference in flux is small. 
\end{Remark}
}

\section{The main result}
{\color{orange} The macromodel \eqref{Intro_Macro_Model} in HMM approximates the effective equation \eqref{Eqn_Intro_Homogenization}. The HMM flux $\bF_{i+1/2,j} \in \mathbb{R}^{d}$ given by \eqref{Intro_HMM_Flux} should then approximate the homogenized flux:
\begin{equation}\label{Homogenized_Flux_Eqn}
\hat{\bF}(\bx_{i+1/2,j}) = A^{0}(\bx)  \nabla \hat{u}(\bx) |_{\bx = \bx_{i+1/2,j}},
\end{equation}
where $\hat{u}(\bx)$ is given in \eqref{Intro_Micro_Problem}. Our main goal here is to prove that if $\hat{u}$ is linear, then HMM captures the homogenized flux in \eqref{Homogenized_Flux_Eqn} up to high orders of accuracy.} {\color{red} Note that since $\hat{u}$ is given by a linear approximation of the coarse scale data, the  flux \eqref{Homogenized_Flux_Eqn} should be interpreted as a local approximation to the exact homogenized flux $\bF^{0}(\bx) = A^{0}(\bx) \nabla u^{0}(\bx)$ in \eqref{Eqn_Intro_Homogenization}. }

{\color{orange} When the media is periodic, i.e., $A = A(\bx/\e)$, the homogenized coefficient $A^{0}$ is a constant matrix and the HMM flux \eqref{Analysis_HMM_Flux} approximates the homogenized flux \eqref{Homogenized_Flux_Eqn} as follows, \cite{Engquist_Holst_Runborg_1}:
\begin{align} \label{PeriodicFlux_Estimate}
\bF  = \hat{\bF}  + O\left( \left( \dfrac{\e}{\eta} \right)^{q+2} \right),
\end{align}
where higher values for $q$ implies better regularity properties for the averaging kernel $K$. In \cite{Arjmand_Runborg_2}, however, various numerical evidence demonstrated that the above rate is no longer valid for locally-periodic coefficients. Moreover, the convergence rate in one-dimension was numerically observed to be different than the rate in higher dimensions. We give the following example to better illustrate the idea.
\begin{Example} Consider the micro problem \eqref{Analysis_MicroProblem_MultiD} in one dimension, with periodic and locally-periodic coefficients
\begin{align}
A(y) & = 1.1 + \dfrac{1}{2}\left( \sin(r_0) + \sin(2\pi y + 2) \right), \nonumber \\ 
A(x,y) & = 1.1 + \dfrac{1}{2}\left( \sin(2 \pi x + r_0) + \sin(2\pi y + 2) \right), \quad r_0 = \frac{1}{10}. \label{1DCoefficients_Example}
\end{align}
The left plot in Figure \ref{Fig_FluxConv1D2D} shows the $O((\e/\eta)^{q+2})$ convergence rate in \eqref{PeriodicFlux_Estimate}, for the periodic coefficient given in \eqref{1DCoefficients_Example}, and $O(\e^{2})$ asymptotic rate for the locally-periodic coefficient. Consider also the two-dimensional coefficients
\begin{align}
A(\by) & = (1.5 + \sin(2\pi y_1))(1.5 + \sin(2 \pi y_2)), \nonumber \\ 
A(\bx,\by) & = (1.5 + \sin(2 \pi y_1) + \sin(2 \pi x_2) \cos(2\pi y_1) ). \label{2DCoefficients_Example}
\end{align}
The right plot in Figure \ref{Fig_FluxConv1D2D}, depicts the $O((\e/\eta)^{q+2})$ convergence rate that the periodic theory predicts, while an $O(\e)$ asymptotic convergence rate is observed for the locally-periodic coefficient in \eqref{2DCoefficients_Example}.
\end{Example}

 \begin{figure}[h] 
    \centering
        \includegraphics[width=0.48\textwidth]{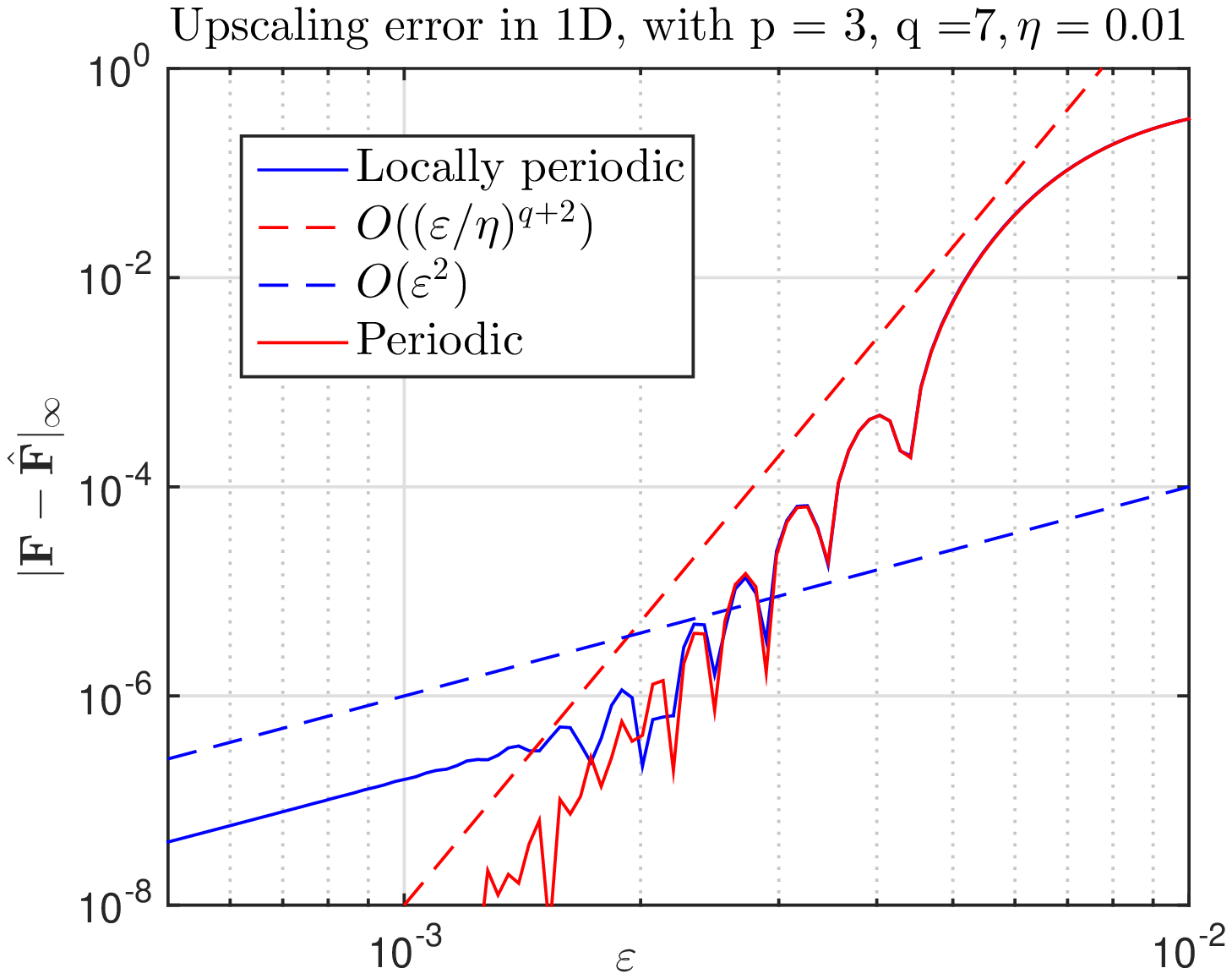}
        \includegraphics[width=0.48\textwidth]{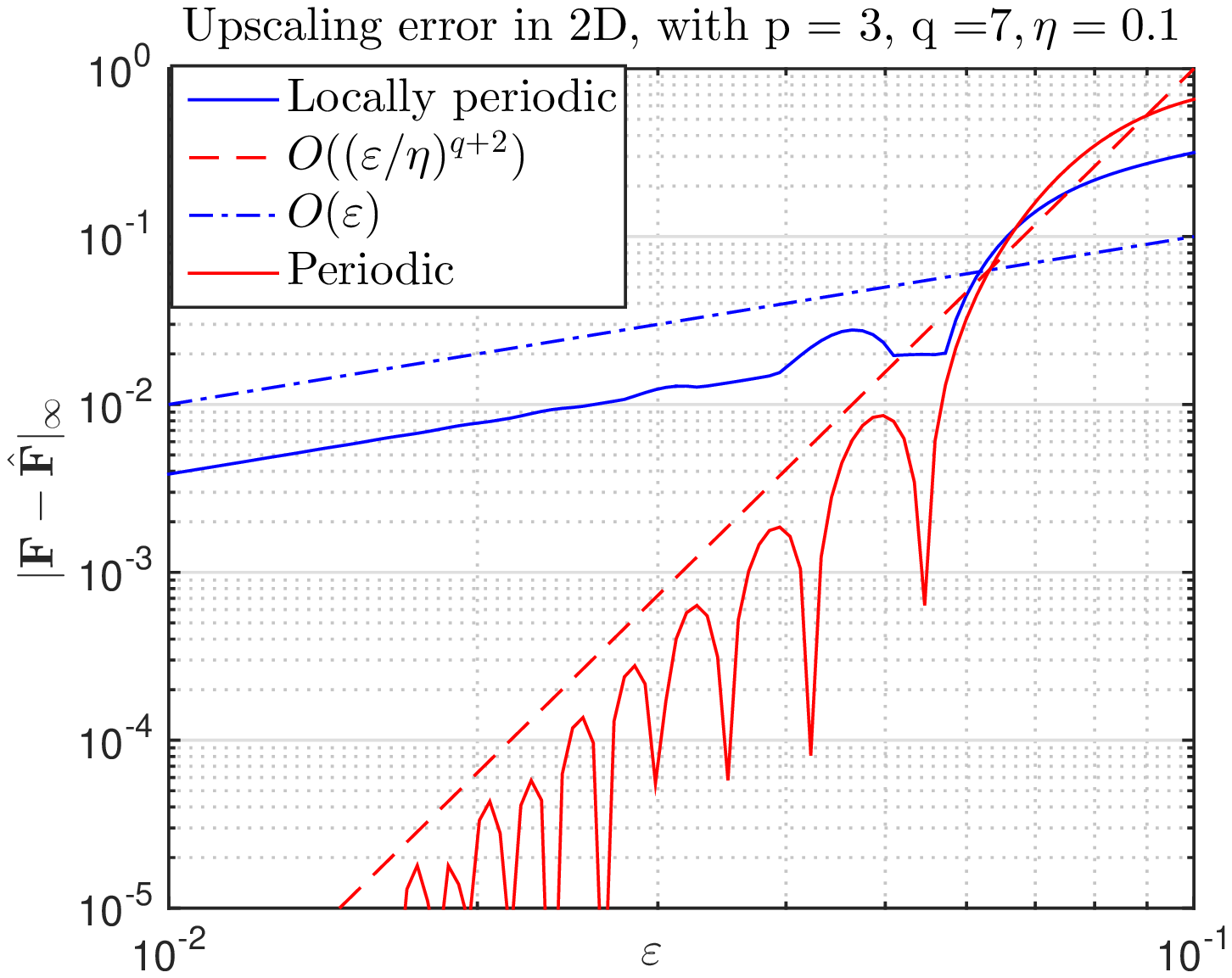}
    \caption{The upscaling errors $| \bF - \hat{\bF}|_\infty$ for periodic and locally-periodic materials in one and two dimensions are depicted. (Left) Upscaling error for the periodic and locally-periodic coefficients \eqref{1DCoefficients_Example} in one dimension: the micro problem \eqref{Analysis_MicroProblem_MultiD} is solved with the parameters $\tau = \eta = 0.01$. The result indicates different convergence rates in periodic and locally-periodic media. (Right) Upscaling error for the periodic and locally-periodic coefficients \eqref{1DCoefficients_Example} in two dimensions: the micro problem \eqref{Analysis_MicroProblem_MultiD} is solved with the parameters $\tau = \eta = 0.1$. The result indicates again different convergence rates in periodic and locally-periodic media. Note also the dependency of the convergence rates, in locally-periodic media, on the dimension. In one dimension, we observe $O(\e^{2})$ asymptotic rate while in two dimensions the convergence rate becomes $O(\e)$.  }
                \label{Fig_FluxConv1D2D}
\end{figure}

In the present paper, we are able to give a rigorous analysis revealing the convergence rates in a multi-dimensional locally-periodic setting. Moreover, our theory fully explains the mentioned dimension-dependent phenomenon. The main result of this paper is the following theorem.
\begin{Theorem}\label{Thm_Main_Thm}Let $\bF$ be given as in \eqref{Analysis_HMM_Flux}, and $\hat{\bF}$ be the homogenized flux \eqref{Homogenized_Flux_Eqn}. Furthermore, let  $\hat{u} =  \bs \cdot \bx$, $K \in \mathbb{K}^{p,q}$ with an even $q$ and $p>1$, and $0<\e \leq \eta=\tau <1$, and assume that $\br_0$ belongs to the compact set $\overline{\Omega}$. Then 
\begin{align*}
\sup_{\br_0 \in \overline{\Omega}}\left| \bF(\br_0) -  \hat{\bF}(\br_0) \right|_{\infty} \leq C  \left|\bs\right|_{\infty} \begin{cases}   \left(\dfrac{\e}{\eta} \right)^{q-1} + \e^{-5} \eta^{7}, & d=1, \\
   \left(\dfrac{\e}{\eta} \right)^{q-1} + \e +\e^{-5} \eta^{7}, & d> 1. 
\end{cases}
\end{align*}
where $C$ does not depend on $\bx, \e, \eta$ but may depend on $K, p, q,d$ or $A$. Moreover, if $\eta = \e^{1-\beta}$ for $0 < \beta < 2/7$, then
\begin{align} \label{Eqn_FHMM_Estimate}
\sup_{\br_0 \in \overline{\Omega}} \left| \bF(\br_0) -  \hat{\bF}(\br_0)  \right|_{\infty} \leq C  |\bs|_{\infty} \begin{cases}  \e^{\beta \left( q-1 \right)} + \e^{2 - 7 \beta} & d=1, \\
 \e^{\beta \left( q-1 \right)} + \e +  \e^{2 - 7 \beta} & d>1.
\end{cases}
\end{align}
\end{Theorem}
\begin{Remark} The HMM error in \eqref{Eqn_FHMM_Estimate} can be made almost equal to $O(\e^{k} + \e^{2})$ for $d=1$ and to $O(\e^{k} + \e)$ for $d>1$, where $k\geq 2$, upon choosing small enough $\beta$ and a large enough $q$, i.e., $q= k/\beta + 1$. Small values for the parameter $\beta$ imply a low computational cost as the size of the micro domains are $\eta = O(\e^{1-\beta}) \approx O(\e)$.
\end{Remark}

\begin{Remark} Note that the simulations in Figure \ref{Fig_FluxConv1D2D} are done for fixed $\eta$ and varying $\e$, and that $O(\e^{2})$ and $O(\e)$ asymptotic rates are observed in one and two dimensions respectively. Theorem \ref{Thm_Main_Thm}, however, suggests an $O(\e^{-5})$ asymptotic error which is not seen in the Figure. This is due to the fact that the error bound in Theorem \ref{Thm_Main_Thm} is not sharp. But note that the right asymptotic rates are recovered upon choosing $\eta = \e^{1-\beta}$. 
\end{Remark}
}

\section{Analysis}
\label{Analysis_Sec}
\subsection{Simplifications}
We start by some simplifications for the analysis. Let $\br_0$ represent a fixed but arbitrary point in the domain $\Omega$. We consider the micro-problem \eqref{Intro_Micro_Problem}, centred at $\br_0$ ($\Omega_{\br_0,\eta} = \br_0 + {\color{blue} [-L_{\eta}, L_{\eta}]^{d}}$), with initial data
$$
\hat{u}(\bx) = \bs \cdot (\bx- \br_0),
$$
where $\bs = \nabla \hat{u}{\color{orange} (\br_0)} \in \mathbb{R}^{d}$ is the slope vector. Note that the constant term $c_0$ introduced in Section \ref{HMM_Sec} is removed from the notation as constant initial data would have zero contribution to the flux. {\color{blue} The starting point of the analysis is to set the local problem \eqref{Intro_Micro_Problem} on $\mathbb{R}^{d}$ and (without loss of generality) shift the micro problem to the origin. }
\begin{align} \label{Analysis_MicroProblem_MultiD}
\partial_{tt} \ueeta(t,\bx) &= \nabla \cdot \left( A(\bx+\br_0 ,(\bx+\br_0)/\e ) \nabla \ueeta(t,\bx)\right),  \text{ in } \mathbb{R}^{d} \times (0,\tau/2], \\
\ueeta(0,\bx) &= \bs \cdot \bx , \quad \partial_t \ueeta(0,\bx) = 0. \nonumber
\end{align}
{\color{blue}  Posing the micro-problem \eqref{Analysis_MicroProblem_MultiD} over $\mathbb{R}^{d}$ is only to simplify the analysis, and does not affect the computational results. This is due to the finite speed of propagation of waves; namely the solution (of the local mixed initial-boundary value problem \eqref{Intro_Micro_Problem}) in the interior region $[0,\tau/2] \times [-\eta/2,\eta/2]^{d}$, which is needed in the upscaling step \eqref{Intro_HMM_Flux},  will be the same as the solution of the corresponding Cauchy problem \eqref{Analysis_MicroProblem_MultiD} since the interior solution will not  be influenced by the periodic conditions of the local problem \eqref{Intro_Micro_Problem} if $L_{\eta} \geq \eta/2 + \tau/2 \sqrt{|A|_{\infty}}$, see \cite{Arjmand_Runborg_1} for more discussions and numerical results.}
Moreover, we rewrite the flux \eqref{Intro_HMM_Flux} as
\begin{equation}\label{Analysis_HMM_Flux}
\bF(\br_0) = \int_{-\tau}^{\tau}\int_{\Omega_{0,\eta}} K_{\eta}(\bx) K_{\tau}(t) A(\bx+\br_0,\bx/\e+\br_0/\e) \nabla \ueeta(t,\bx) \; d\bx \; dt. 
\end{equation} 
We introduce also
$$
A_{\br_0,\gamma}(\bx,\bx/\e)  := A(\bx+\br_0,\bx/\e + \gamma). 
$$
As $A$ is periodic in the second argument, we can replace the coefficient $A$ in \eqref{Analysis_MicroProblem_MultiD} by $A_{\br_0,\gamma}$ evaluated at $\gamma =  \{ \br_0/\e \}$, where $\{ \ba  \}$ denotes the fractional part of $\ba \in \mathbb{R}^{d}$. 

{\color{orange} \subsection{Outline of the analysis}

\begin{itemize}
\item[Step 1.] \emph{Expansion:} In this part, the solution is scaled as $u^{\e,\eta}(t,\bx) = \e v(t/\e,\bx/\e)$, and an asymptotic expansion is used to express $v$ as
$$
v(t,\bx) \approx  v_{0}(t,\bx) + \e v_{1}(t,\bx)  + \dfrac{\e^2}{2} v_{2}(t,\bx) + \ldots.
$$ 
The main result of the section is to estimate the difference between $\tilde{u}_1^{\e,\eta} :=  \e v_0(t/\e,\bx/\e) + \e^{2}v_1(t/\e,\bx/\e)$, and $u^{\e,\eta}$ the solution of the micro problem \eqref{Analysis_MicroProblem_MultiD}.
\item[Step 2.] \emph{Quasi-polynomials:} The main aim in this part is to write $v_0(t,\bx)$ and $v_{1}(t,\bx)$ in terms of simpler (periodic) functions. Namely, we show that 
$$
v_0(t,\by) = \bs \cdot \by  + v_{00}(t,\by), \text{ and }  v_{1}(t,\by)  = v_{10}(t,\by)  + \sum_{j=1}^{d} y_j v_{11j}(t,\by),
$$
where $v_{00}(t,\cdot), v_{10}(t,\cdot), \{ v_{11j}(t,\cdot) \}_{j=1}^{d}$ are periodic functions on $Y:=(0,1]^d$.
\item[Step 3.] \emph{Energy estimates:}  This part includes energy estimates for elliptic and second order hyperbolic PDEs. The results of this section is used later in Step 4.
\item[Step 4.] \emph{Time averages:} In the upscaling step \eqref{Intro_HMM_Flux}, the averaging operator in time can be treated separately from the spatial averaging operator due to linearity, see Subsection \ref{SubSec_TimeAverages}. The main aim is to write down explicit equations for the temporally averaged quantities defined by $d_{00}(\by):= K_{\tau} \ast v_{00}(\cdot/\e,\by)$, and for $d_{11j}$, and $d_{10}$ (which are defined similarly).  
\item[Step 5.] \emph{Decomposition of the flux:} In this part, the HMM flux \eqref{Intro_HMM_Flux} is decomposed as 
\begin{align*} 
\bF(\br_0) &= \left( \mathcal{K}_{\tau,\eta} \ast A(\cdot,\cdot/\e)\nabla\tilde{u}^{\e,\eta}(\cdot,\cdot) \right)(0,0) + \underbrace{\left( \mathcal{K}_{\tau,\eta} \ast A(\cdot,\cdot/\e)\left( \nabla u^{\e,\eta}-\nabla\tilde{u}_1^{\e,\eta} \right) \right)(0,0)}_{\mathcal{E}_{tail}}  \nonumber, 
\end{align*}  
and the HMM flux is expressed as $\bF = \bF_{0} + \e \bF_{1} + \delta + \mathcal{E}_{tail}$, where 
\begin{align*}
\bF_{0}(\br_0) &= \left( K_{\eta} \ast  A(\cdot,\cdot /\e)\left(\bs  + \nabla_\by d_{00}(\cdot/\e) \right) \right) (0) \\
\bF_{1}(\br_0) &= \left( K_{\eta} \ast  A(\cdot,\cdot /\e)\left(\nabla_\by d_{10}(\cdot/\e)  + d_{11}(\cdot/\e)\right) \right)(0) \\
\delta(\br_0) &= \left(  K_{\eta}\ast \sum_{j=1}^{d}x_{j} A(\cdot,\cdot/\e) \nabla_\by d_{11j}(\cdot/\e) \right)(0).
\end{align*}
Moreover, the result from Step 2 is used to estimate the tail $\mathcal{E}_{tail}$.
\item[Step 6.] \emph{The main proof}: Here it is proved that $|\bF - \bF_{0}| \leq C (\e/\eta)^{q}$, and that the terms $\e \bF_{1}$ and $\delta$ are bounded and small, and the final estimate is obtained. It is also proved that in one-dimension the result $|\bF_{1}| \leq C (\e/\eta)^{q-1}$ holds; explaining the one-dimensional effect seen in the numerics.
\end{itemize}

}
\subsection{Expansion}\label{SubSec_Expansion}
We consider now the micro problem \eqref{Analysis_MicroProblem_MultiD} {\color{blue} which is} posed over $\mathbb{R}^d$. By the scaling $u^{\e,\eta}(\e t,\e \bx) := \e v(t,\bx;\e,\br_0,\gamma)$ we have
\begin{align} \label{eqn_ScaledMainProblem_2D}
\partial_{tt} v(t,\bx;\e,\br_0,\gamma) = \nabla \cdot \left( A_{\br_0,\gamma}(\e \bx,\bx) \nabla v(t,\bx;\e,\br_0,\gamma)\right), \\
v(0,\bx;\e,\br_0,\gamma) = \bs \cdot \bx, \quad \partial_t v(0,\bx;\e,\br_0,\gamma) = 0. \nonumber
\end{align}
Now we define $v_{k}(t,\bx;\br_0,\gamma):=\partial_{\e}^{k} v(t,\bx;0,\br_0,\gamma)$, and consider the expansion
\begin{equation}\label{Asymptotic_Expansion_2D}
v(t,\bx;\e,\br_0,\gamma) =  v_{0}(t,\bx;\br_0,\gamma) + \e v_{1}(t,\bx;\br_0,\gamma)  + \dfrac{\e^2}{2} v_{2}(t,\bx;\br_0,\gamma) {\color{red} + R(t,\bx;\e,\br_0,\gamma).}
\end{equation}
For $k=0$, we have
\begin{align} 
\partial_{tt} v_{0}(t,\bx;\br_0,\gamma) = \nabla \cdot \left( A_{\br_0,\gamma}(0,\bx) \nabla v_{0}(t,\bx;\br_0,\gamma)\right), \nonumber \\
v_{0}(0,\bx;\br_0,\gamma) = \bs \cdot \bx, \quad \partial_t v_{0}(0,\bx;\br_0,\gamma) = 0. \label{Eq_v0}
\end{align}
Let $v_0 = \bs \cdot \bx  + v_{00}(t,\bx;\br_0,\gamma)$, then
\begin{align} 
\partial_{tt} v_{00}(t,\bx;\br_0,\gamma) &= \nabla \cdot \left( A_{\br_0,\gamma}(0,\bx) \nabla v_{00}(t,\bx;\br_0,\gamma)\right) +  \nabla \cdot A_{\br_0,\gamma}(0,\bx) \bs  \label{Eq_v00} \\
v_{00}(0,\bx;\br_0,\gamma) &= \partial_t v_{00}(0,\bx;\br_0,\gamma) = 0. \nonumber
\end{align}
For the higher order terms we have the following relation
\begin{align} \label{Eq_vm}
\partial_{tt} v_{m}(t,\bx;\br_0,\gamma) &= \nabla \cdot \left(A_{\br_0,\gamma}(0,\bx) \nabla  v_{m}(t,\bx;\br_0,\gamma) \right) + G_{m}(t,\bx;\br_0,\gamma) \nonumber \\
v_{m}(0,\bx;\br_0,\gamma) &= \partial_t v_{m}(0,\bx;\br_0,\gamma) = 0, \quad m>0,
\end{align}
where $G_m$ is, with $c_{jm} = \dbinom{m}{j}$,
\begin{equation*}
G_m(t,\bx;\br_0,\gamma) = \nabla \cdot \left( \sum_{j=0}^{m-1} c_{jm}\left( \partial_{\e}^{m-j} A_{\br_0,\gamma}(\e \bx , \bx)\right)|_{\e =0}  \nabla v_j(t,\bx;\br_0,\gamma)  \right).
\end{equation*}

\begin{figure}[ht]
  \centering 
    \includegraphics[width=0.55\textwidth]{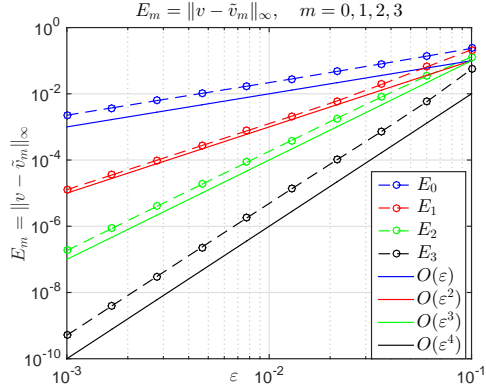}
    \caption{Convergence as $\e \longrightarrow 0$ of the truncated expansion $\tilde{v}_m = \sum_{k=1}^{m} \frac{\e^k}{k!} v_k$. The problem \eqref{Eq_vm} is solved for $m=0,1,2,3$ in one-dimension with $A(x,y) = 1.1 + 0.5 (\sin(2 \pi x)  + \sin(2\pi y))$ in the domain $(t,x) \in [0,1] \times [-3,3]$ with periodic boundary conditions. The error is defined as $E_m = \max_{(t,x) \in [0,1] \times [-1.5,1.5]} | v - \tilde{v}_m|$. } \label{Fig_VmConvergence}
\end{figure}
\begin{Remark} We drop $\br_0$ and $\gamma$ in the notation for $A_{\br_0,\gamma}(0,\by)$, $v_{j}(t,\by;\br_0,\gamma)$ and $v_{jk}(t,\by;\br_0,\gamma)$ and simply write $A(0,\by), v_j(t,\by)$ and $v_{jk}(t,\by)$. Moreover, in addition to the assumptions on $A$ given in the introduction, we assume that $\br_0$ belongs to the compact set $\overline{\Omega}$ so that all the constants are uniform in $\br_0$.
\end{Remark}

Now we present a theorem to estimate the tail of the expansion in \eqref{Asymptotic_Expansion_2D}. 
\begin{Theorem}\label{Truncated_Convergence_Thm} Assume that $\Omega_{L} = [-L,L]^d$, and that $A \in C^{\infty}(\Omega_L \times Y)$. Let $v(t,\bx;\varepsilon)$ be the solution of \eqref{eqn_ScaledMainProblem_2D} and $\tilde{v}_m(t,\bx;\varepsilon)$ be defined by

\begin{equation}\label{Eqn_Truncated_vm}
\tilde{v}_m(t,\bx;\varepsilon)  = \sum_{k=0}^m \frac{\varepsilon^k}{k!} v_k(t,\bx),
\end{equation}
where $v_k$ is given by \eqref{Eq_v0} and \eqref{Eq_vm} for $k=0$ and $k>1$ respectively. Then for any $M > L  + t \sqrt{ \left| A \right|_{\infty}}$
\begin{equation*}
\begin{array}{ll}
\| \nabla \left( v -\tilde{v}_m\right)(t,\cdot;\e) \|_{L^2(\Omega_L)} &\leq \displaystyle C \varepsilon^{m+1} \sum_{j=0}^{m} M^{m-j} \left(1 + M \right) t \max_{|z|\leq t} \left\| v_{j}(z,\cdot) \right\|_{H^2(\Omega_M)},
\end{array}
\end{equation*}
where $C$ does not depend on $t,\e$ and $M$ but may depend on $A,m,d$. Moreover, for $m=1$, the above estimate gives

\begin{equation} \label{Eqn_Tail_Estimate_v_m_1}
\begin{array}{ll}
\| \nabla \left( v -\tilde{v}_1\right)(t,\cdot;\e)  \|_{L^2(\Omega_L)} &\leq C \varepsilon^{2} \left(  1+ M^{d/2+1} \right) \left( 1 + M^{2} \right) \left( 1 + t^4 \right)  \left| \bs \right|_{\infty}.    
\end{array}
\end{equation}
\end{Theorem}
\begin{proof} See the appendix.
\end{proof}
To illustrate the result of Theorem \ref{Truncated_Convergence_Thm}, we present a numerical test in Figure \ref{Fig_VmConvergence}, which shows the claimed rate of convergence. Note that $M$ and the final time $t$ are fixed in the simulation.
\begin{Corollary} \label{Lemma_Tail_Estimate} Let $\ueeta(t,\bx)$ be the microscopic solution solving \eqref{Analysis_MicroProblem_MultiD} and let $\tilde{u}^{\e,\eta}_1(t,\bx)$ $:=\e \tilde{v}_1(t/\e,\bx/\e;\e)$, where $\tilde{v}_1$ is given by \eqref{Eqn_Truncated_vm} with $m=1$. Moreover, assume that $\e \leq \tau$ and $\e \leq \eta$. Then, with $e^{\e,\eta}(t,\bx):=u^{\e,\eta}(t,\bx)  - \tilde{u}^{\e,\eta}_1(t,\bx)$, the estimate \eqref{Eqn_Tail_Estimate_v_m_1}  implies
\begin{equation*}
\begin{array}{ll}
\displaystyle \sup_{t \in (0,\tau]} \| \nabla_{\bx} e^{\e,\eta}(t,\cdot) \|_{L^2(\Omega_\eta)} &\leq C \varepsilon^{2+d/2} \left( \left( \dfrac{\eta}{\e}\right)^{d/2+3} \left( \dfrac{\tau}{\e}\right)^{4}  + \left( \dfrac{\eta}{\e}\right)^{d/2+7} \right) \left| \bs \right|_{\infty},
\end{array}
\end{equation*}
where $C$ does not depend on $\e,\eta,\tau$ but may depend on $A$ and $d$. 
\end{Corollary}
\begin{proof} Let $e := v  - \tilde{v}_1$ and $e^{\e,\eta} =  u^{\e,\eta}  - \tilde{u}^{\e}_1$, then $e^{\e,\eta}(t,\bx) = \e e(t/\e,\bx/\e;\e)$ and with $\by=\bx/\e$ the chain rule gives
$$
\nabla e^{\e,\eta}(t,\bx) = \nabla_\by e(t/\e, \bx/\e).
$$
Then with $\alpha = \e/\eta$ we get
\begin{align*}
\sup_{t \in \left( 0,\tau \right]} \| \nabla e^{\e,\eta} \|_{L^2(\Omega_{\eta})}^{2} &= \sup_{t \in \left( 0,\tau \right]} \int_{\Omega_{\eta}} \left| \nabla e^{\e,\eta}(t,\bx)  \right|^2 d\bx =  \sup_{t \in \left( 0,\tau \right]} \int_{\Omega_{\eta}} \left| \nabla_\by e(t/\e,\bx/\e)  \right|^2 d\bx \\
& \hspace{-1cm}  = \e^{d}  \sup_{t \in \left( 0,\tau \right]} \int_{\Omega_{\alpha^{-1}}} \left| \nabla_\by e(t/\e,\by)  \right|^2 d\by  =  \e^{d}  {\color{red} \sup_{t \in \left( 0,\tau/\e \right]}  \left\| \nabla_{\by} e(t,\cdot)\right\|_{L^2(\Omega_{\alpha^{-1}})}^2.}
\end{align*}
The final result follows by exploiting the last inequality and by putting $t=\tau/\e$ and $M = \alpha^{-1}$ in estimate \eqref{Eqn_Tail_Estimate_v_m_1}. 
\end{proof}
\subsection{Quasi-polynomials}
From an analysis point of view it is desirable to deal with purely periodic functions. Unfortunately, the terms $v_{m}$ in \eqref{Eq_vm} are not periodic. However, they possess a nice structure known as the \emph{quasi-polynomials} where the coefficients of usual polynomials are replaced by $Y$-periodic smooth functions.  
\begin{Definition}\label{Def_QuasiPolynomials} A function $P(\bx,\by): \mathbb{R^{d} \times \mathbb{R}^{d} \longrightarrow \mathbb{R}}$ belongs to the set $\mathbb{P}_n$ of quasi-polynomials of degree $n$ if
$$
P(\bx,\by) = \sum_{|\beta| \leq n} p_{\beta}(\by) \bx^{\beta},
$$
where $\beta$ represents a multi-index so that $\bx^{\beta} = x_{1}^{\beta_1} x_{2}^{\beta_2} \cdots x_{d}^{\beta_d}$, $|\beta| = \sum_{j=1}^{d} \beta_j $, and $P_{\beta} \in C^{\infty}(Y)$ are infinitely differentiable $1$-periodic functions, named the coefficient functions of $P$. 
\end{Definition}
Now we will state a lemma which shows that, in general, periodic wave equations with quasi-polynomial data have quasi-polynomial solutions. For this let us define $Q,Z,F(t,\cdot,\cdot) \in \mathbb{P}_n$ so that
$$
Q(\bx,\by) =  \sum_{|\beta| \leq n} q_{\beta}(\by) \bx^{\beta}, \quad Z(\bx,\by)= \sum_{|\beta|\leq n} z_{\beta}(\by) \bx^{\beta}, \quad F(t,\bx,\by) = \sum_{|\beta|\leq n} f_{\beta}(t,\by) \bx^{\beta}, 
$$
and consider the wave equation {\color{red} with a uniformly elliptic and bounded periodic coefficient $B \in C^{\infty}(Y)$:}
\begin{align} \label{Eq_u}
\partial_{tt} u = \nabla \cdot \left(B(\bx) \nabla  u(t,\bx) \right) + F(t,\bx,\bx) \nonumber \\
u(0,\bx) = Q(\bx,\bx), \quad \partial_t u(0,\bx) = Z(\bx,\bx).
\end{align}
{\color{red} Note that the solution $u$ is not bounded in usual $H^1$ spaces (not a weak solution). However, it is a classical solution since it satisfies the wave equation pointwise everywhere.} The following lemma states that the solution $u$ is a quasi-polynomial of degree $n$ as well. 
\begin{Lemma}\label{Lemma_QuasiPolynomials} There is a family of quasi-polynomials $U(t,\cdot,\cdot) \in \mathbb{P}_n$ such that the solution to \eqref{Eq_u} is given as $u(t,\bx) = U(t,\bx,\bx) = \sum_{|\beta|\leq n} u_{\beta}(t,\bx) \bx^{\beta} $. The coefficient functions of $U$ solve the forced periodic wave equations
 \begin{align*} 
\partial_{tt} u_{\beta}(t,\bx) = \nabla \cdot \left(B(\bx) \nabla  u_{\beta}(t,\bx) \right) + p_{\beta} + f_{\beta} \nonumber \\
u_{\beta}(0,\bx) = q_{\beta}, \quad \partial_t u_{\beta}(0,\bx) = z_{\beta},
\end{align*}
where $f_{\beta}, q_{\beta}, z_{\beta}$ are the coefficient functions of $F,Q,Z,$ and 
\begin{equation*}
p_{\beta}(t,\bx) =  
\begin{cases}
0, & |\beta | = n, \\
M[u_{\beta +1}], & |\beta | = n-1, \\
M[u_{\beta + 1}] + N[u_{\beta + 2}], & |\beta | \leq n-2,
\end{cases}
\end{equation*}
where {\color{orange} (with $\{ e_j \}_{j=1}^{d}$ being the standard canonical basis vectors in $\mathbb{R}^d$)}
\begin{align*}
M[u_{\beta +1 }] &:= \sum_{j=1}^{d} \left( \beta_j + 1 \right) \left( \nabla \cdot \left( B e_j u_{\beta + e_j} \right) + e_j^T B \nabla u_{\beta + e_j} \right) \\
N[u_{\beta +2}] &:= \sum_{i,j=1}^{d} \left( \beta_i + 1 \right)\left( \beta_j + 1 \right) e_{i}^T B e_j u_{\beta + e_i + e_j} \left(1 - \delta_{ij} \right) \\ &+ \sum_{i=1}^{d} \left( \beta_i + 1 \right)\left( \beta_i + 2 \right) e_{i}^T B e_i u_{\beta + 2 e_i}.
\end{align*}
\end{Lemma}
\begin{proof} The proof of this Lemma in one-dimension was given in \cite{Arjmand_Runborg_1}. Here, we give the general proof. Let $L=: \nabla \cdot B \nabla$. We first note that
\begin{equation*}
\begin{array}{ll}
 &   L[w_{\beta}(\bx)\bx^\beta] =  \displaystyle
  \sum_{i,j=1}^d \partial_{x_i} \left( B_{ij}(\bx)\bx^\beta \partial_{x_j} w_{\beta}(x) \right) +
  \sum_{i,j =1}^d \beta_j \partial_{x_i} \left( \bx^{\beta-e_j}B_{ij}(\bx)w_{\beta}(\bx) \right) \\
   &= 
 \displaystyle   \bx^{\beta} \sum_{i,j=1}^d \partial_{x_i} \left( B_{ij}(\bx) \partial_{x_j} w_{\beta}(\bx) \right) +
\sum_{i,j= 1}^d \beta_j \bx^{\beta - e_j} \left( B_{ji} \partial_{x_i} w_{\beta}(\bx)  + \partial_{x_i} \left( B_{ij} w_{\beta}(\bx) \right)\right) \\ &+
  \displaystyle \sum_{i,j= 1}^d  \left( \left( 1 - \delta_{ij }\right) \beta_i \beta_j  + \delta_{ij} \beta_i\left(\beta_{j}-1\right)\right) \bx^{\beta - e_i - e_j} B_{ij} w_{\beta}(\bx).
\end{array}
\end{equation*}

\noindent From this it follows that
\begin{equation*}
\begin{array}{ll}
&  L[U(t,\bx,\bx)] =  \displaystyle \sum_{|\beta| \leq n} \bx^{\beta} L[u_{\beta}] + \sum_{|\beta| \leq n}  \sum_{i,j= 1}^d \beta_j \bx^{\beta - e_j} \left( B_{ji} \partial_{x_i} u_{\beta}  + \partial_{x_i} \left( B_{ij} u_{\beta} \right)\right)  \\ &+ \displaystyle  \sum_{|\beta| \leq n}  \sum_{i,j= 1}^d  \left( \left( 1 - \delta_{ij }\right) \beta_i \beta_j  + \delta_{ij} \beta_i \left(\beta_{j}-1\right)\right) \bx^{\beta - e_i - e_j} B_{ij} u_{\beta}
\\
 &=  \displaystyle \sum_{|\beta| \leq n} \bx^{\beta} L[u_{\beta}]  + \sum_{|\beta| \leq n-1}  \bx^{\beta} \sum_{j= 1}^d \left(\beta_j  + 1 \right)  \sum_{i=1}^d \left( B_{ji} \partial_{x_i} u_{\beta + e_j}  + \partial_{x_i} \left( B_{ij} u_{\beta + e_j} \right)\right) 
\\ 
 
&+ \displaystyle \sum_{|\beta| \leq n-2}   \bx^{\beta} \sum_{i,j= 1}^d  \left( \left( 1 - \delta_{ij }\right) (\beta_i+1) (\beta_j + 1)  + \delta_{ij} (\beta_i+2) \left(\beta_{i}+1\right)\right) B_{ij} u_{\beta + e_i  + e_j} \\  
&= \displaystyle \sum_{|\beta| \leq n} \bx^{\beta} L[u_{\beta}]  +  \sum_{|\beta| \leq n-1}  x^{\beta} M[u_{\beta+1}] +  \displaystyle \sum_{|\beta| \leq n-2}   \bx^{\beta} N[u_{\beta+2}] \\   &=
 \displaystyle \sum_{|\beta| \leq n} \bx^{\beta} \left(  L[u_{\beta}]  + p_{\beta} \right). 
 \end{array}
\end{equation*}

\noindent On the other hand
\begin{equation*}
  \partial_{tt} U(t,\bx,\bx) = \sum_{|\beta|\leq n}\partial_{tt} u_\beta \bx^\beta
  =\sum_{|\beta|\leq n}
  \bx^\beta(L[u_\beta] + p_\beta + f_\beta)
    = L[U] + F.
\end{equation*}
Similarly, the initial data agrees and $U(t,\bx,\bx)$ is therefore a solution.
\end{proof}

{\color{orange} We will now use Lemma \ref{Lemma_QuasiPolynomials} to express the solutions of \eqref{Eq_vm} as quasi-polynomials.} When $m=1$, equation \eqref{Eq_vm} reads  
\begin{equation}\label{eqn_v1_MultiD}
\partial_{tt} v_{1}(t,\by)  = L[v_1] + \sum_{j=1}^{d} \nabla_{\by} \cdot \left(  \left( y_j \partial_{x_j}A(0,\by) \right) \nabla_{\by} v_{0} \right). 
\end{equation}
Recall from \eqref{Eq_v0} that $v_0$ solves 
\begin{align*} 
\partial_{tt} v_{0}(t,\by) = \nabla \cdot \left( A(0,\by) \nabla v_{0}(t,\by)\right), \nonumber \\
v_{0}(0,\by) = \bs \cdot \by, \quad \partial_t v_{0}(0,\by) = 0.
\end{align*}
Since $\nabla_\by v_0 \in \mathbb{P}_0$, the forcing term of equation \eqref{eqn_v1_MultiD} will be a quasi-polynomial of degree one. Therefore by  Lemma \ref{Lemma_QuasiPolynomials} we can write
\begin{equation}\label{eqn1_Lemma_v1}
v_1(t,\by)  = v_{10}(t,\by)  +  \sum_{j=1}^{d} y_j v_{11j}(t,\by),
\end{equation}
where $v_{10}$ and $v_{11j}$ are periodic in $\by$, and
\begin{align}\label{eqn2_Lemma_v1}
\partial_{tt} v_{11k} &= L[v_{11k}]  +  f_{11k}, \quad k=1,\ldots, d \nonumber \\
\partial_{tt} v_{10} &= L[v_{10}] + M[v_{11}] + f_{10},
\end{align}
with
\begin{align}\label{eqn3_Lemma_v1}
f_{11k} =\nabla_{\by} \cdot \left( \partial_{x_k} A \bs \right)+ \nabla_{\by} \cdot \left( \left( \partial_{x_k} A \right)  \nabla_{\by} v_{00} \right)&, \quad 
f_{10} = \nabla_{\bx} \cdot A \bs  + \left( \nabla_{\bx} \cdot A \right) \cdot \nabla_\by v_{00} \nonumber \\
M[v_{11}]= \sum_{j=1}^{d} \nabla_\by \cdot \left( A e_{j} v_{11j} \right) &+ \sum_{j=1}^{d }e_{j}^{T} A \nabla_\by v_{11j}.
\end{align}

 \begin{figure}[h] 
    \centering
        \includegraphics[width=0.48\textwidth]{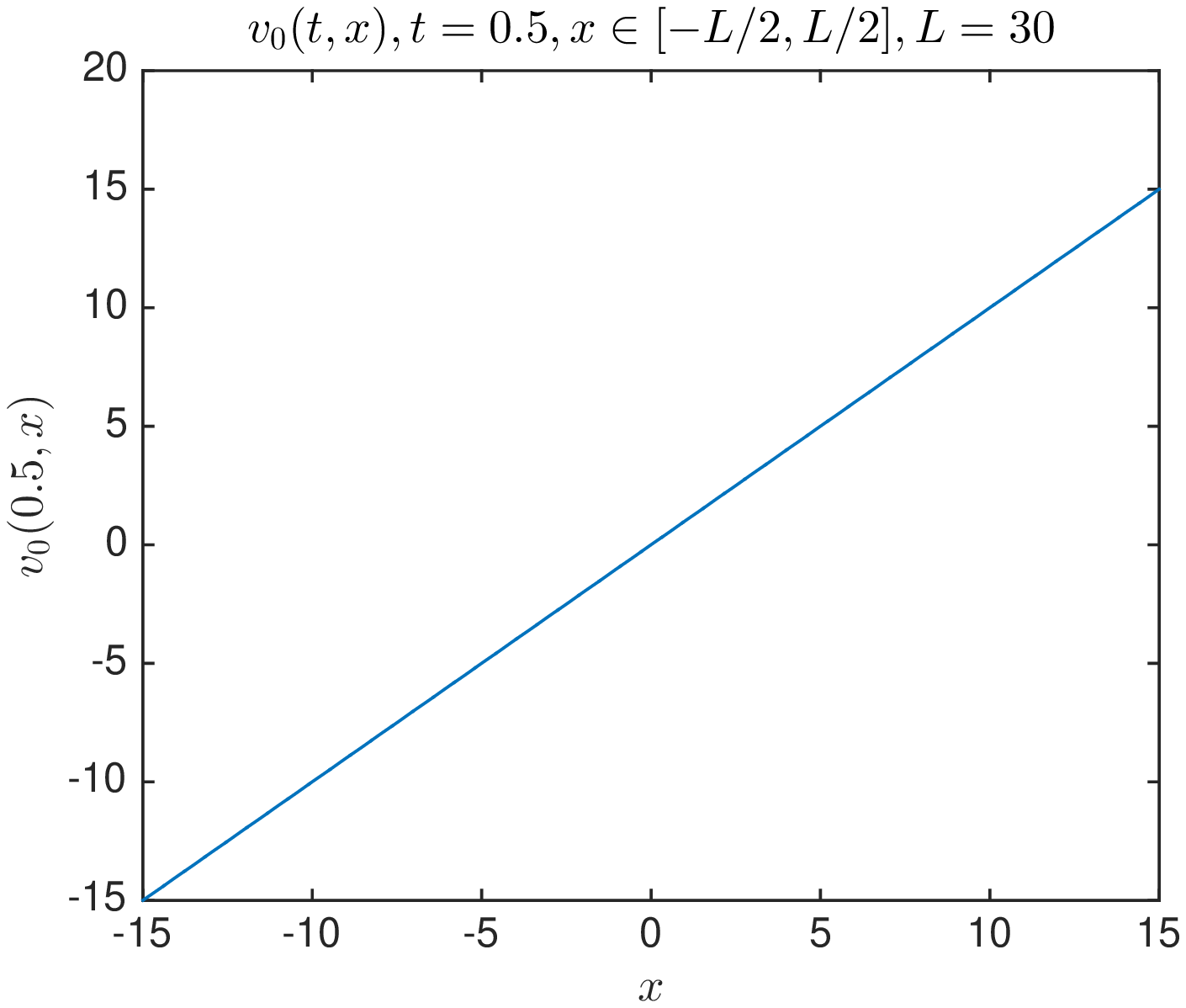}
        \includegraphics[width=0.48\textwidth]{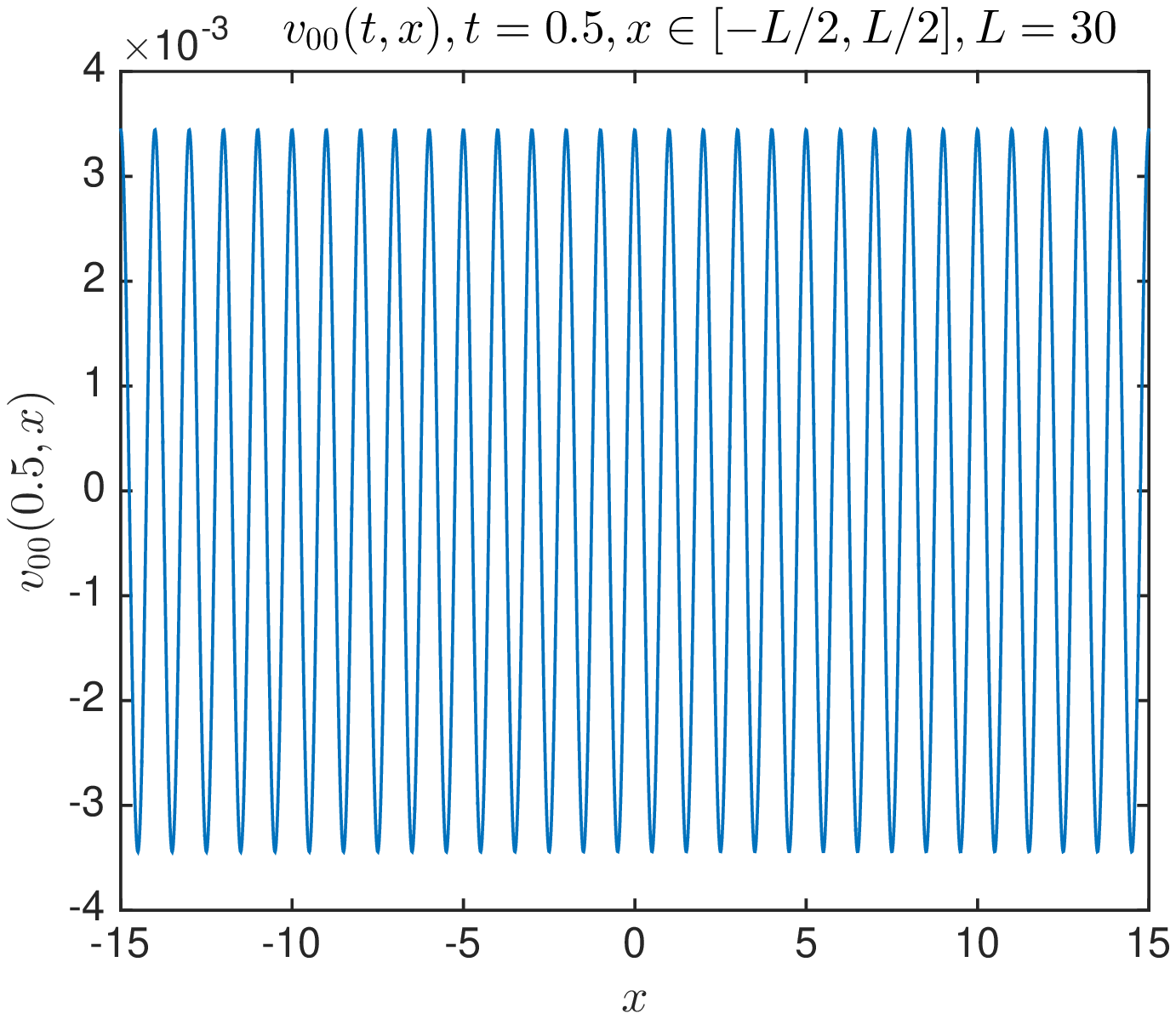}
    \caption{The equation \eqref{Eq_v0} is solved using a one dimensional coefficient $A(x,y) = 1.1 + \dfrac{1}{2}\left( \sin(2 \pi x + 0.1) + \sin(2\pi y + 2) \right)$ and the initial data $v_0(0,x) = x$ over the domain $(t,x) \in (0,1] \times [-L,L]$ for $L = 30$. (left) the term $v_0(t,\cdot)$ is depicted for $t = 0.5$. (right) the term $v_{00}(t,\cdot)$ is computed by $v_{00} = v_0  - x$. }
                \label{Fig_v0v00Terms}
\end{figure}

 \begin{figure}[h] 
    \centering
        \includegraphics[width=0.47\textwidth]{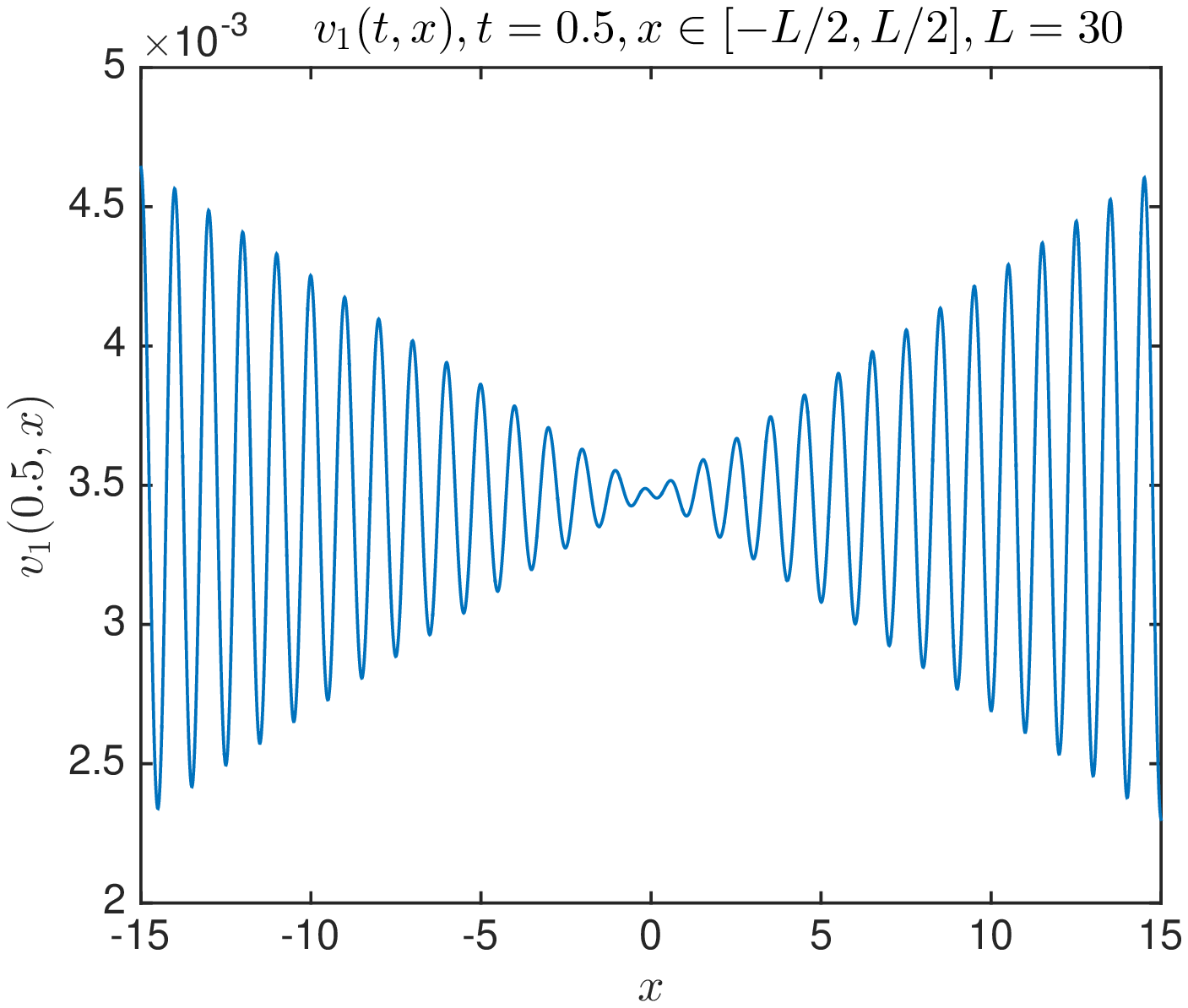}
        \includegraphics[width=0.47\textwidth]{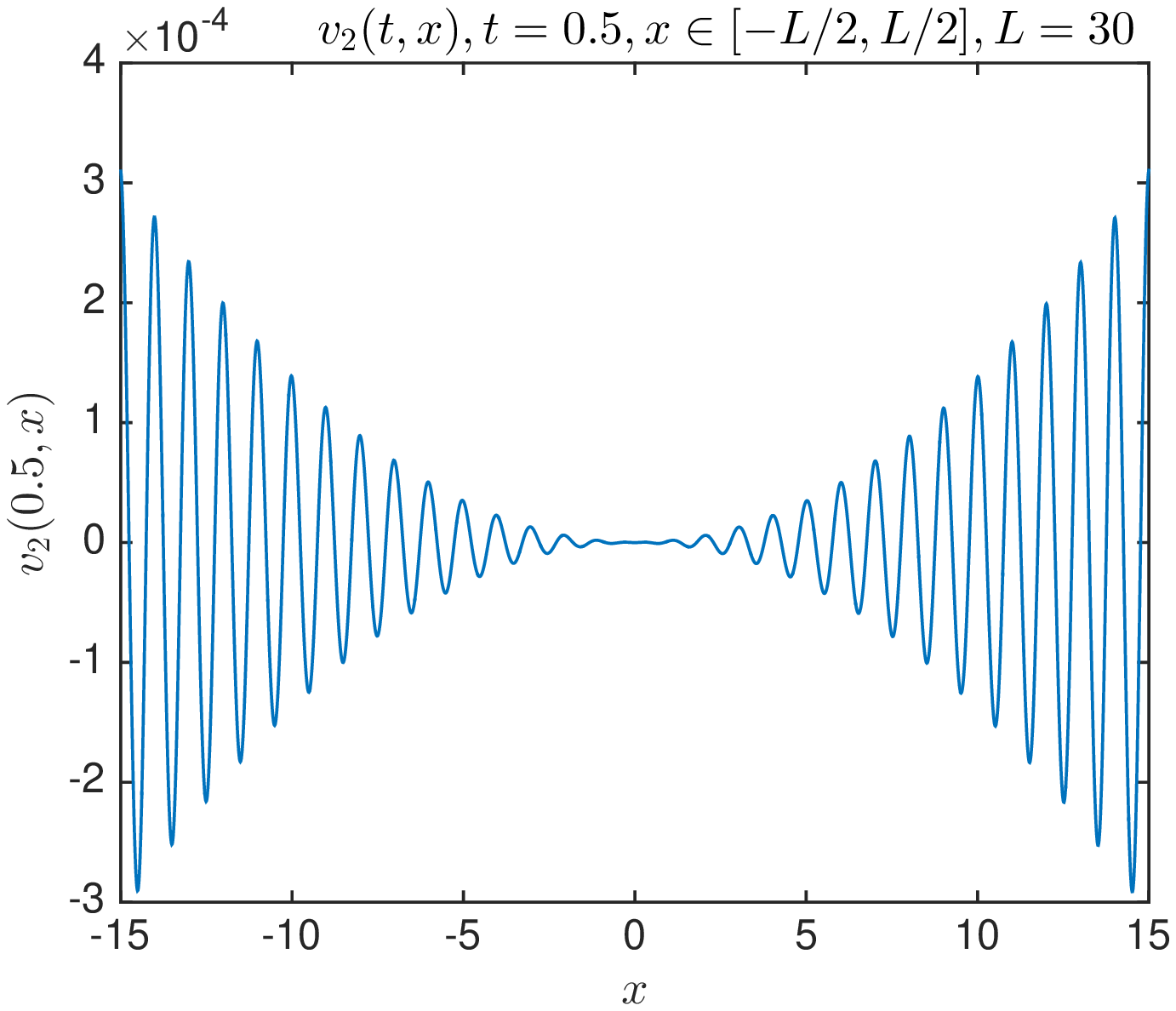}
    \caption{The equations \eqref{Eq_vm} for $m= 1$ and $m=2$  are solved with the same coefficient and domain as in Figure \ref{Fig_v0v00Terms}. (left) this plot shows that $v_1 \in \mathbb{P}_1$ has linear growth in space as the theory of quasi-polynomials predicts. (right) this plot shows that $v_2 \in \mathbb{P}_2$ has a quadratic growth in space.}
                \label{Fig_v1v2Terms}
\end{figure}
\begin{Remark}\label{Rem_Locally_Periodic} Using the notation introduced in the beginning of Section $4$, the derivative $\partial_{x_j} A(0,\by)$ should be understood as $\partial_{x_j} A_{\br_0,\gamma}(\bx,\by) |_{\bx=0}$. 
\end{Remark}
To illustrate the idea, we present also numerical simulations, in Figures \ref{Fig_v0v00Terms} and \ref{Fig_v1v2Terms}, where we solve \eqref{Eq_v0}, and \eqref{Eq_vm} for $m= 1$ and $m=2$. The numerical results are consistent with the fact that $v_0 \in \mathbb{P}_1$, $v_1 \in \mathbb{P}_1$. Moreover, we observe, in Figure \ref{Fig_v1v2Terms}, that $v_2 \in \mathbb{P}_2$ has a quadratic growth in spatial dimensions.
\subsection{Energy estimates}\label{SubSec_EnergyEstimates}
In this section we present energy estimates for elliptic and second order hyperbolic equations. Throughout the section, we will assume $Y$-periodic {\color{red} uniformly elliptic and bounded coefficients} of the form $B = B(\by)$, where $B_{ij} \in C^{\infty}(\overline{Y})$. The first lemma concerns the regularity in high order Sobolev norms of solutions to periodic elliptic problems. In this section, we write
$$
L = \nabla_\by \cdot B \nabla_\by,
$$
and we denote the averages over the unit cube $Y$ by
$$
\overline{f} := \int_{Y} f(\by) \; d\by.
$$
\begin{Lemma}(Lemma $3.1$ in \cite{Arjmand_Runborg_1})\label{Lemma_Elliptic_Regularity} Suppose that $f \in H^{k}(Y)$ is a $Y$-periodic function and that $\overline{f}= 0$. Then, for any positive integer $n$, there exists a unique periodic function $u \in H^{k+2n}(Y)$ satisfying
$$
(-1)^{n} L^{n}[u] = f, \quad \overline{u}=0.
$$ 
In addition, the following stability estimate holds
\begin{align}\label{Eqn_Elliptic_Estimates}
\| u \|_{H^{k+2n}(Y)} \leq C \| f \|_{H^k(Y)},
\end{align}
where $C$ does not depend on $f$ but may depend on $k,Y$ and $B$.
\end{Lemma}

Now we present a lemma which gives energy estimates for the periodic wave equation in high order Sobolev norms. 
\begin{Lemma}\label{Lemma_Energy_Estimate}
Let $f \in C^{\infty}([0,T]\times Y)$, $f(t,\cdot)$ be $Y$-periodic, and $\overline{f(t,\cdot)} = 0$. Moreover, let $g,h \in C^{\infty}(Y)$ be $Y$-periodic functions with $\bar{g} = \bar{h} = 0$. Then there is a unique solution $u \in C^{\infty}([0,T]\times Y)$, with $\overline{u(t,\cdot)}= 0$, solving
\begin{align*}
\partial_{tt} u  = L[u]  + f(t,\by), \\
u(0,\by) = g, \quad \partial_t u(0,\by) = h.
\end{align*}
Moreover, there exists a constant $C$ independent of $t$ such that for any $n\geq 0$
\begin{equation}\label{Eqn_HighOrderEstimate_Wave}
\| u(t,\cdot) \|_{H^{2n + 1}(Y)} \leq C E^{1/2}_{L^n[u]}(0) + C \begin{cases} \int_{0}^{t} \| f(z,\cdot) \|_{H^{2n}(Y)} dz, & f \text{ is time dependent} \\
\| f \|_{H^{2n}(Y)}, & f \text{ is time independent},
\end{cases}
\end{equation}
where the energy is defined as
$$
E_u(t):= \dfrac{1}{2} \int_{Y} |\partial_t u(t,\by)|^2 + {\color{red} B} \nabla u \cdot \nabla u(t,\by) \; d\by. 
$$
\end{Lemma}
\begin{proof}The result is classical for $n=0$. To get the estimate for $n>0$, we consider $w := L^n[u]$. Then by elliptic regularity, Lemma \ref{Lemma_Elliptic_Regularity}, we have 
\begin{equation} \label{Eqn_u_w_Estimate}
\| u \|_{H^{k + 2n}(Y)} \leq  \| w \|_{H^{k}(Y)}. 
\end{equation}
Next we apply the operator $L^{n}$ to the main equation and obtain
\begin{align*}
\partial_{tt} w(t,y)  &= L[w]  + L^{n}[f], \\
w(0,y) = L^n[g]&, \quad \partial_t w(0,\by) = L^n[h].
\end{align*}
Then we have
\begin{align*}
\| w(t,\cdot) \|_{H^{1}(Y)} &\leq C E^{1/2}_{w}(0) + C \begin{cases} \int_{0}^{t} \| L^n[f](z,\cdot) \|_{L^{2}(Y)} dz, & f \text{ is time dependent} \\
\| L^n[f] \|_{L^{2}(Y)}, & f \text{ is time independent},
\end{cases}
\end{align*}
The final result follows by observing that $\| L^n[f] \|_{L^2(Y)} \leq C \| f \|_{H^{2n}(Y)}$, and using the estimate \eqref{Eqn_u_w_Estimate} with $k=1$.
\end{proof}
\subsection{Time averages} \label{SubSec_TimeAverages}
The upscaling step in HMM includes averaging of the microscopic flux over micro boxes in time and space, see \eqref{Analysis_HMM_Flux}. Since 
$$
\left( K_{\tau} \ast \nabla u^{\e,\eta} \right)(t,\bx) = \nabla \left( K_{\tau} \ast u^{\e,\eta} \right)(t,\bx),
$$
and since $A$ is time independent we can write the HMM flux in \eqref{Analysis_HMM_Flux} as 
$$
\bF(\br_0) = \left( K_{\eta}\ast A^{\e} \nabla \left( K_{\tau} \ast u^{\e,\eta} \right) \right)(0,0).
$$
The idea is now to write down equations for the time average $\left( K_{\tau}\ast u^{\e}(\cdot,\bx)\right)(0)$. To do this we start by presenting some intermediate results. First we use a theorem from \cite{Arjmand_Runborg_1} which is used to derive equations for the local time averages of solutions to periodic wave equations. For the next theorem, we will assume coefficient functions of the form $B = B(\bx,\by)$, where $B_{ij} \in C^{\infty}(\overline{\Omega} \times \overline{Y})$, and $B(\bx,\cdot)$ is $Y$-periodic. We write also $L:= \nabla_\by \cdot B \nabla_\by$.
\begin{Theorem} (Theorem $4.1$ in \cite{Arjmand_Runborg_1}) \label{Thm_Time_Averaging_Wave_Equation}Suppose $\alpha = \frac{\varepsilon}{\tau}$ where $0< \varepsilon \leq \tau$. Let $f \in C^{\infty}([0,\alpha^{-1}],Y)$ be a $Y$-periodic function with $\overline{f(t,\cdot)} = 0$, and $K \in \mathbb{K}^{p,q}$ with an even $q$. Furthermore, assume that $w(t,\by;\bx)$ is the solution of the periodic wave equation parametrized by $\bx$, i.e., $B = B(\bx,\by)$,
\begin{equation}\label{Main_PDE_Eqn}
\begin{array}{ll}
\partial_{tt} w(t,\by;\bx) =  L[w]  + f(t,\by;\bx), \\
w(0,\by;\bx) = \partial_t w(0,\by;\bx)  = 0. 
\end{array}
\end{equation}
Then the local time average $ d^{\{2k \}}(\bx,\by):= K_{\tau} \ast \partial_t^{2k} w(\cdot/\varepsilon,\by;\bx)(0)$  satisfies 
\begin{equation*}
L[ d^{\{ 2k\}}] = -\sum_{\ell=k}^{q/2 - 1} L^{-\ell + k} K_{\tau}  \ast \partial_t^{2\ell} f(\cdot/\varepsilon,\by;\bx)(0)  + \alpha^q R_k(\bx,\by), \quad  k=0,1,\cdots,q/2 -1,
\end{equation*}
where $R_k(\bx,\by)$ is  $Y$-periodic with zero average ($\overline{R_k(\bx,\cdot)} = 0$), and  
\begin{align}\label{Eqn_R_Estimates}
\| R_k(\bx,\cdot) \|_{H^1(Y)} &\leq C \max_{|t| \leq 1}  \| w(t/\alpha, \cdot;\bx) \|_{L^2(Y)},
\end{align}
where $C$ does not depend on $\alpha,\e,\eta$ but may depend on $Y,K,p$ or $q$ .
\end{Theorem}
\begin{Remark} Our main aim in this section is to apply Theorem \ref{Thm_Time_Averaging_Wave_Equation} to equations \eqref{Eq_v00}, and \eqref{eqn2_Lemma_v1} in order to be able to find equations for the local time averages. However, the theorem assumes periodic wave equations with forcing terms that have zero average over the unit cube $Y$. Clearly, the equations for $v_{00}$ and $\{ v_{11j} \}_{j=1}^{d}$ satisfy this condition. On the other hand, the equation for $v_{10}$ needs special treatment since its right hand side does not have zero average. To handle this problem we introduce $g(t):= \overline{v_{10}(t,\cdot)}$, and split $v_{10}(t,\by)$ into two parts as follows.
$$
v_{10}(t,\by) = \tilde{v}_{10}(t,\by)  + \overline{v_{10}(t,\cdot)} = \tilde{v}_{10}(t,\by)  + g(t).
$$
Then $\tilde{v}_{10}$ satisfies a wave equation with zero average forcing,
\begin{equation}\label{Eqn_vtilde10}
\partial_{tt} \tilde{v}_{10}(t,\by) = L[\tilde{v}_{10}] +  \tilde{M}[v_{11}] + \tilde{f}_{10}(t,\by),
\end{equation}
where 
$$
\tilde{M}[v_{11}] := M[v_{11}] - \overline{M[v_{11}]}, \quad \tilde{f}_{10}(t,\by) = f_{10}(t,\by)  - \overline{f_{10}(t,\cdot)}.
$$  
and $M[v_{11}]$ is given as in \eqref{eqn3_Lemma_v1}. Moreover, the function $g(t)$ solves the second order ODE:
\begin{align} \label{Eqn_g_Equation}
g^{\prime\prime}(t) &= \overline{M[v_{11}]} + \overline{f_{10}(t,\cdot)} \nonumber \\
g(0) &= g^{\prime}(0) =0.
\end{align}

\end{Remark}
We introduce the local time averages $d_{00}$, $\{d_{11j}\}_{j=1}^{d}$, and $d_{10}$ 
\begin{equation} \label{Eqn_LocalTimeAverages}
\begin{array}{cc}
d_{00}(\by) :=  K_{\tau} \ast v_{00}(\cdot/\e,\by)(0),  \quad d_{11j}(\by) :=  K_{\tau} \ast v_{11j}(\cdot/\e,\by)(0), \\  d_{10}(\by) :=  K_{\tau} \ast \tilde{v}_{10}(\cdot/\e,\by)(0).
\end{array}
\end{equation}
In the following theorem we use Theorem \ref{Thm_Time_Averaging_Wave_Equation} to derive equations for these local averages. 
\begin{Theorem}\label{Thm_TimeAverages_d} Suppose $v_{00}$ and $v_{11j}$ solve \eqref{Eq_v00}, and \eqref{eqn2_Lemma_v1} respectively. Moreover, suppose $\tilde{v}_{10}$ solves \eqref{Eqn_vtilde10} and let $K \in \mathbb{K}^{p,q}$, $\alpha = \frac{\e}{\eta}$ with $0< \e \leq \eta$ and $\tau = \eta$. Then the local time averages defined in \eqref{Eqn_LocalTimeAverages} satisfy
\begin{align} \label{Eqn_TimeAverages_d}
L[d_{00}] &= -\nabla_\by \cdot A \bs  + \alpha^{q} Z_{00}(\by) \nonumber \\
L[d_{11j}] &= -\nabla_\by \cdot \partial_{x_j} A \bs - \nabla_\by \cdot \left( \left( \partial_{x_j} A \right) \nabla_\by d_{00}\right) + \alpha^{q} Z_{11j}(\by) ,  \nonumber \\
L[d_{10}] &= -\sum_{j=1}^{d} \nabla_\by \cdot \left( A e_{j} d_{11j} \right) - \sum_{j=1}^{d} e_{j}^{T} A \nabla_\by d_{11j} - \nabla_{\bx} \cdot A \bs  - \left( \nabla_{\bx} \cdot A \right) \cdot \nabla_{\by} d_{00} \nonumber \\ \hspace{1cm} &+ \tilde{K}  + \alpha^{q} Z_{10}(\by), 
\end{align}
where $\tilde{K}$ is a constant such that the right hand side has zero average, $Z_{00},Z_{11j},$ and $Z_{10}$ are $Y$-periodic functions with zero average. Moreover, 
\begin{equation}\label{Eqn_Z_Estimates}
\| Z \|_{H^1(Y)} \leq C_1 |\bs |_{\infty} \begin{cases} 1, &  Z = Z_{00} \\
\alpha^{-1}, & Z = Z_{11j} \\
\alpha^{-2}, & Z = Z_{10}. 
\end{cases}
\end{equation}
where $C_1$ is a constant independent of $\varepsilon$, $\eta$, $\alpha$.
\end{Theorem}
\begin{proof} We prove this theorem in two steps
\begin{enumerate}
\item{Energy estimates: First we will prove the following energy estimates}
\begin{align} \label{Eqn_Energy_Estimates_v}
\max_{|r| \leq 1} \| w(r/\alpha, \cdot) \|_{H^{2n+1}(Y)} \leq C_1 | \bs |_{\infty}  \begin{cases}
1, & w = v_{00} \\
\alpha^{-1}, & w=v_{11j} \\
\alpha^{-2}, & w = \tilde{v}_{10}. 
\end{cases}
\end{align}
An application of Lemma \ref{Lemma_Energy_Estimate} to equation \eqref{Eq_v00} gives
\begin{equation*}
\max_{|r|\leq 1} \| v_{00}(r/\alpha, \cdot) \|_{H^{2n+1}(Y)} \leq C \| \nabla_\by \cdot A \bs \|_{H^{2n}(Y)} \leq C |\bs|_{\infty}, \quad \forall \quad n \geq 0.
\end{equation*}
For the term $v_{11j}$, we first note that with $I := \displaystyle \max_{|z| \leq \alpha^{-1}}  \| f_{11j}(z,\cdot) \|_{H^{2n}(Y)} $
\begin{align*} 
\begin{array}{lll}
\displaystyle I \leq C \max_{|z|\leq \alpha^{-1}}  \left(  \| \nabla_\by \cdot \partial_{x_j} A \bs \|_{H^{2n}(Y)} + \| \nabla_\by \cdot  \left( \left(\partial_{x_j} A \right)  \nabla_{\by} v_{00}(z,\cdot) \right)\|_{H^{2n}(Y)} \right) \\
\displaystyle \leq  C_1 |\bs|_{\infty} + C_2 \max_{|z|\leq \alpha^{-1}} \|  v_{00}(z,\cdot) \|_{H^{2n+2}(Y)} \leq C |\bs|_{\infty}, \quad \forall \quad n \geq 0.
\end{array}
\end{align*}
We then apply Lemma \ref{Lemma_Energy_Estimate} to the first equation in \eqref{eqn2_Lemma_v1}. This gives
\begin{align*}
\max_{|r| \leq 1}  \| v_{11j}(r/\alpha,\cdot) \|_{H^1(Y)}   &\leq C \max_{|r|\leq 1} \int_{0}^{r/\alpha} \| f_{11j}(z,\cdot) \|_{L^2(Y)} dz \leq C \alpha^{-1} |\bs|_{\infty}.
\end{align*}
Now we give an estimate for $\tilde{v}_{10}$. For this, we first recall from \eqref{Eqn_vtilde10} and \eqref{eqn3_Lemma_v1} that
$$
\tilde{M}[v_{11}] := M[v_{11}] - \overline{M[v_{11}]}, \quad \tilde{f}_{10}(t,\by) = f_{10}(t,\by)  - \overline{f_{10}(t,\cdot)},
$$  
where
\begin{align*}
f_{10} &= \nabla_{\bx} \cdot A \bs  + \left( \nabla_{\bx} \cdot A \right) \cdot \nabla_\by v_{00} \nonumber \\
M[v_{11}]&= \sum_{j=1}^{d} \nabla_\by \cdot \left( A e_{j} v_{11j} \right) + \sum_{j=1}^{d }e_{j}^{T} A \nabla_\by v_{11j}.
\end{align*}
It follows that
\begin{align} \label{Eqn_f10_Estimate}
\max_{|r|\leq 1}   \| f_{10}(r/\alpha,\cdot)\|_{H^{2n}(Y)}  &\leq C \left( \|  \nabla_\bx \cdot A \bs\|_{H^{2n}(Y)}  + \max_{|r|\leq 1}   \| v_{00}(r/\alpha,\cdot)\|_{H^{2n+1}(Y)}   \right) \\ &\leq C |\bs|_{\infty}, \nonumber
\end{align}
and
\begin{align*}
\max_{|r|\leq 1}   \| \tilde{f}_{10}(r/\alpha,\cdot)\|_{H^{2n}(Y)}  &\leq \max_{|r|\leq 1}   \| f_{10}(r/\alpha,\cdot)\|_{H^{2n}(Y)}  + \max_{|r|\leq 1}   \| \overline{f_{10}(r/\alpha,\cdot)}\|_{H^{2n}(Y)} \\
&\leq 2  \max_{|r|\leq 1}   \|f_{10}(r/\alpha,\cdot)\|_{H^{2n}(Y)} \leq C |\bs|_{\infty}.
\end{align*}
The last inequality is due to the estimate for $v_{00}$ in \eqref{Eqn_Energy_Estimates_v} and the inequality \eqref{Eqn_f10_Estimate}. Moreover, we use the estimate for $v_{11j}$ in \eqref{Eqn_Energy_Estimates_v} to see that
\begin{align*}
\max_{|r|\leq 1} \| \tilde{M}[v_{11}](r/\alpha,\cdot)\|_{H^{2n}(Y)}  &\leq C \max_{|r|\leq 1} \sum_{j=1}^{d} \| \tilde{M}[v_{11j}](r/\alpha,\cdot)\|_{H^{2n+1}(Y)} \leq C \alpha^{-1}|\bs|_{\infty}.
\end{align*}
We then apply Lemma \ref{Lemma_Energy_Estimate} to \eqref{Eqn_vtilde10} for $II :=\max_{|r| \leq 1}  \| \tilde{v}_{10}(r/\alpha,\cdot) \|_{H^{2n+1}(Y)}$,
\begin{align*}
II   &\leq C \max_{|r|\leq 1} \int_{0}^{r/\alpha} \| \tilde{M}[v_{11}](z,\cdot)\|_{H^{2n}(Y)} + \| \tilde{f}_{10}(z,\cdot)\|_{H^{2n}(Y)} \; dz \\
&\leq C \alpha^{-1} \left(  \max_{|r|\leq 1} \| \tilde{M}[v_{11}](r/\alpha,\cdot)\|_{H^{2n}(Y)} + \max_{|r|\leq 1}   \| \tilde{f}_{10}(r/\alpha,\cdot)\|_{H^{2n}(Y)}  \right)\leq C \alpha^{-2} |\bs|_{\infty}.
\end{align*}
\item{Equations for time averages:} To derive the first equation in \eqref{Eqn_TimeAverages_d} we apply Theorem \ref{Thm_Time_Averaging_Wave_Equation} to \eqref{Eq_v00}. We immediately see that, since the forcing $\nabla_\by \cdot A\bs$ is time-independent, 
\begin{align*}
L[d_{00}] &= -\sum_{\ell=0}^{q/2 - 1} L^{-\ell} K_{\tau} \ast \partial_t^{2\ell} \left( \nabla_\by \cdot A \bs \right) + \alpha^{q} R_{00,0}(\by) \\
&=  -\nabla_\by \cdot A \bs + \alpha^{q} R_{00,0}(\by).
\end{align*}
We define now $Z_{00}:= R_{00,0}$. Then by Theorem \ref{Thm_Time_Averaging_Wave_Equation} and step $1$ we obtain 
$$
\| Z_{00} \|_{H^1(Y)} \leq C \max_{|r|\leq 1} \| v_{00}(r/\alpha,\cdot) \|_{L^2(Y)} \leq C_1 |\bs|_{\infty}. 
$$
Now we focus on $d_{11j}$. In this case, we apply Theorem \ref{Thm_Time_Averaging_Wave_Equation} to \eqref{eqn2_Lemma_v1}. This gives
\begin{align*}
L[d_{11j}] &= -\sum_{\ell=0}^{q/2-1} L^{-\ell} K_{\tau} \ast \partial_t^{2\ell} f_{11j}(\cdot/\e,\by)(0) + \alpha^{q} R_{11j,0}(\by) \\
&= -\sum_{\ell=0}^{q/2-1} L^{-\ell} K_{\tau} \ast \partial_t^{2\ell} \left( \nabla_\by \cdot \partial_{x_j} A\bs  + \nabla_\by \cdot \left( \left( \partial_{x_j}A\right) \nabla_\by v_{00}\right) \right) + \alpha^q R_{11j,0}(\by) \\
&= -\nabla_\by \cdot \partial_{x_j} A\bs - \nabla_\by \cdot \left( \left( \partial_{x_j}A\right) \nabla_\by d_{00}\right) +  \alpha^{q} Z_{11j}(\by)
\end{align*}
where we use the notation $d_{00}^{\{2\ell \}}(\by) := K_{\tau} \ast \partial_t^{2\ell} v(\cdot/\e,\by)$ from Theorem \ref{Thm_Time_Averaging_Wave_Equation} and define
$$
Z_{11j}:= -\alpha^{-q} \sum_{\ell=1}^{q/2-1} L^{-\ell} \nabla_\by \cdot  \left( \left( \partial_{x_j}A\right) \nabla_\by d_{00}^{\{ 2\ell \}} \right)+ R_{11j,0}(\by).
$$
Now we estimate $Z_{11j}$. From step $1$ and by Theorem \ref{Thm_Time_Averaging_Wave_Equation} we have 
\begin{equation} \label{Bound_R11j}
\| R_{11j,0} \|_{H^1(Y)} \leq C \max_{|r|\leq 1} \| v_{11j}(r/\alpha, \cdot) \|_{L^2(Y)} \leq C |\bs|_{\infty} \alpha^{-1}.
\end{equation}
Furthermore, let 
\begin{equation} \label{Eqn_psi0}
\psi_{\ell}^{0}:=L^{-\ell} \nabla_y \cdot  \left( \left( \partial_{x_j}A\right) \nabla_\by d_{00}^{\{ 2\ell \}} \right), 
\end{equation}
then by Lemma \ref{Lemma_Elliptic_Regularity} we have
\begin{equation} \label{Estimate_psi0}
\| \psi_{\ell}^0 \|_{H^{2\ell}(Y)} \leq C \| \nabla_\by \cdot \left( \left( \partial_{x_j} A \right) \nabla_\by d_{00}^{\{2\ell\}} \right) \|_{L^2(Y)} \leq C \|d_{00}^{\{2\ell\}} \|_{H^2(Y)}.
\end{equation}
On the other hand, by Theorem \ref{Thm_Time_Averaging_Wave_Equation} and elliptic regularity we obtain
\begin{equation} \label{Eqn_d002l}
L[d_{00}^{\{2\ell\}}]  = \alpha^{q} R_{00,\ell}, 
\end{equation}
and
\begin{align}\label{Estimate_d00}
\|d_{00}^{\{2\ell\}} \|_{H^2(Y)} \leq C \alpha^q \| R_{00,\ell} \|_{L^2(Y)} \leq C \alpha^q \max_{|r|\leq 1} \| v_{00}(r/\alpha,\cdot) \|_{H^1(Y)} \leq C \alpha^q |\bs|_{\infty}.
\end{align}
Therefore $\| \psi_{\ell}^0\|_{H^2(Y)} \leq C \alpha^q |\bs|_{\infty}$, and it follows that
\begin{align*}
\| Z_{11j} \|_{H^1(Y)} \leq  \| R_{11j,0} \|_{H^1(Y)} + \alpha^{-q} \sum_{\ell=1}^{q/2 - 1} \| \psi_{\ell} \|_{H^1(Y)}  \leq C_1 \alpha^{-1} |\bs|_{\infty}.
\end{align*}
Now we concentrate on the term $d_{10}$. First we introduce short hand notations:
\begin{align*}
\psi_{\ell}^{1} &= L^{-\ell} \left( \left( \nabla_\bx \cdot A \right) \cdot \nabla_\by d_{00}^{\{2\ell \}} - \overline{\left( \nabla_\bx \cdot A \right) \cdot \nabla_\by d_{00}^{2\ell}}   \right) \\
\psi_{\ell}^{2} &= L^{-\ell} \left( M[d_{11}^{\{  2\ell\} }]  - \overline{M[d_{11}^{\{  2\ell\} }]} \right),
\end{align*}
where the $M[d_{11}]$ is understood as replacing $v$ with $d$ in \eqref{eqn3_Lemma_v1}. Then by Theorem \ref{Thm_Time_Averaging_Wave_Equation} we have 
\begin{align*}
L[d_{10}] &= -\sum_{\ell=0}^{q/2-1} L^{-\ell} K_{\tau} \ast \partial_t^{2\ell} \left( \tilde{f}_{10}(\cdot/\e,\by)(0) + \tilde{M}[v_{11}(\cdot/\e, \by)] \right)+ \alpha^{q} R_{10}(\by) \\
&= -\nabla_\bx \cdot A \bs - \left( \nabla_\bx \cdot A\right) \cdot \nabla_\by d_{00} - M[d_{11}] + K  + \alpha^{q} Z_{10}(\by),
\end{align*}
where $K$ and $Z_{10}$ are defined as
\begin{align} \label{Eqn_Z10}
K&:= \overline{\nabla_\bx \cdot A \bs   + \left( \nabla_\bx \cdot A\right) \cdot \nabla_\by d_{00}  + M[d_{11}] }, \nonumber \\
Z_{10}(\by) &:=  R_{10}(\by) + \alpha^{-q} \sum_{\ell=1}^{q/2-1} \left( \psi_{\ell}^{1} +  \psi_{\ell}^{2}\right). 
\end{align}
In order to bound $Z_{10}$, we first see from Theorem \ref{Thm_Time_Averaging_Wave_Equation} and the first part in this proof that
\begin{align*}
\|  R_{10}  \|_{H^1(Y)} \leq  C \max_{|r| \leq 1} \| \tilde{v}_{10}(r/\alpha,\cdot)  \|_{L^2(Y)} \leq C \alpha^{-2} |\bs|_{\infty}.
\end{align*}
Next by Lemma \ref{Lemma_Elliptic_Regularity}, we have
\begin{align*}
\|  \psi_{\ell}^{1}  \|_{H^{2\ell}(Y)} &\leq  C  \left\|  \left( \nabla_\bx \cdot A \right) \cdot \nabla_\by d_{00}^{\{2\ell \}} - \overline{\left( \nabla_\bx \cdot A \right) \cdot \nabla_\by d_{00}^{\{2\ell \}}}   \right\|_{L^2(Y)}  \\ 
&\leq C \left\| d_{00}^{\{ 2\ell \}}  \right\|_{H^2(Y)} \leq C \alpha^{q} |\bs|_{\infty}.
\end{align*}
and
\begin{align*}
\|  \psi_{\ell}^{2}  \|_{H^{2\ell}(Y)} \leq  C  \left\|   M[d_{11}^{\{  2\ell\} }]  - \overline{M[d_{11}^{\{  2\ell\} }]}  \right\|_{L^2(Y)} \leq C \sum_{j=1}^{d} \left\| d_{11j}^{\{ 2\ell \}}  \right\|_{H^1(Y)}. 
\end{align*}
But by Theorem \ref{Thm_Time_Averaging_Wave_Equation} and equation \eqref{Eqn_psi0} we have 
\begin{align*}
L[d_{11j}^{\{ 2 \ell \} }] = L^{\ell} \sum_{k=\ell}^{q/2-1} \psi_{k}^{0} + \alpha^{q} R_{11j,\ell},
\end{align*}
where similar to \eqref{Bound_R11j}, it holds that
\begin{align*}
\| R_{11j,\ell} \|_{H^1(Y)} \leq C |\bs|_{\infty} \alpha^{-1}.
\end{align*}
We apply now Lemma \ref{Lemma_Elliptic_Regularity}, use the fact that $\ell \leq k$ and obtain
\begin{equation*}
\begin{array}{lll}
\left\|  d_{11j}^{\{ 2\ell \}}   \right\|_{H^1(Y)} &\leq& C_1 \sum_{k=\ell}^{q/2-1} \left\|  L^{\ell}  \psi_{k}^{0}  \right\|_{L^2(Y)} + C_2 \alpha^{q-1} |\bs|_{\infty}  \\ &\leq& C_1 \sum_{k=\ell}^{q/2-1} \left\|  \psi_{k}^{0}  \right\|_{H^{2\ell}(Y)}  + C_2 \alpha^{q-1} |\bs|_{\infty} \\ &\leq& C_1 \sum_{k=\ell}^{q/2-1} \left\|  \psi_{k}^{0}  \right\|_{H^{2k}(Y)}  + C_2 \alpha^{q-1} |\bs|_{\infty} \leq C \alpha^{q-1} |\bs|_{\infty}. 
\end{array}
\end{equation*}
Note that we used estimates \eqref{Estimate_psi0} and \eqref{Estimate_d00} in the last step. Therefore,
\begin{align*}
\| Z_{10} \|_{H^1(Y)} \leq C \alpha^{-2} |\bs|_{\infty}.
\end{align*}
\end{enumerate}
\end{proof}

\subsection{Decomposition of the flux}
We consider first a truncated version of the expansion \eqref{Asymptotic_Expansion_2D} by taking only $v_0$ and $v_1$ into consideration. We denote the truncated microscopic solution by $\tilde{u}^{\e,\eta}(t,\bx)$. Then by the scaling introduced in Subsection \ref{SubSec_Expansion} we have
\begin{align}\label{Eqn_TruncatedExpansion}
\tilde{u}^{\e,\eta}(t,\bx) &= \e v_0(t/\e,\bx/\e) + \e^2 v_1(t/\e,\bx/\e) \nonumber \\
&=  \left( \bs \cdot \bx  + \e v_{00}(t/\e,\bx/\e) \right) + \e^2 v_{10}(t/\e,\bx/\e) + \e \sum_{j=1}^{d} x_j v_{11j}(t/\e,\bx/\e). 
\end{align}
Using the notation $A(\bx,\by) = A_{\br_0,\gamma}(\bx,\by)$, the HMM flux in \eqref{Analysis_HMM_Flux} can be written as
\begin{align} \label{Flux_HMM_Split}
\bF(\br_0) &= \left( \mathcal{K}_{\tau,\eta} \ast A(\cdot,\cdot/\e)\nabla\tilde{u}^{\e,\eta}(\cdot,\cdot) \right)(0,0) + \underbrace{\left( \mathcal{K}_{\tau,\eta} \ast A(\cdot,\cdot/\e)\left( \nabla u^{\e,\eta}-\nabla\tilde{u}^{\e,\eta} \right) \right)(0,0)}_{\mathcal{E}_{tail}}  \nonumber  \\ &= \left(  K_{\eta} \ast A(\cdot,\cdot/\e) \nabla \underbrace{K_{\tau} \ast \tilde{u}^{\e,\eta}(\cdot,\cdot)}_{:=d^{\e}(x)}\right)(0,0) + \mathcal{E}_{tail}.
\end{align}  
On the other hand, we apply $K_{\tau}$ to \eqref{Eqn_TruncatedExpansion} and obtain
\begin{align*}
d^{\e}(\bx) = \bs \cdot \bx  +  \e d_{00}( \bx/\e)  + \e^{2} d_{10}( \bx/\e) + \e^{2} \left(K_{\tau} \ast g(\cdot/\e)\right)(0) +  \e \sum_{j=1}^{d} x_j d_{11j}( \bx/\e).
\end{align*}
Then we can rewrite the HMM flux as
\begin{align}\label{Eqn_HMM_Flux_Decomposed}
\bF = \bF_{0} + \e \bF_{1} + \delta + \mathcal{E}_{tail},
\end{align}
where 
\begin{align*}
\bF_{0}(\br_0) &= \left( K_{\eta} \ast  A(\cdot,\cdot /\e)\left(\bs  + \nabla_\by d_{00}(\cdot/\e) \right) \right) (0) \\
\bF_{1}(\br_0) &= \left( K_{\eta} \ast  A(\cdot,\cdot /\e)\left(\nabla_\by d_{10}(\cdot/\e)  + d_{11}(\cdot/\e)\right) \right)(0) \\
\delta(\br_0) &= \left(  K_{\eta}\ast \sum_{j=1}^{d}x_{j} A(\cdot,\cdot/\e) \nabla_\by d_{11j}(\cdot/\e) \right)(0).
\end{align*}
We will first bound the tail $\mathcal{E}_{tail}$ in the following lemma.
\begin{Lemma}\label{Lemma_Tail_Flux_Estimate}
Suppose that $\mathcal{E}_{tail}$ is defined as in \eqref{Flux_HMM_Split}, with $ 0 < \e \leq \eta$ and $\tau = \eta$. Then
$$
\left| \mathcal{E}_{tail}  \right|\leq C \e^{-5} \eta^{7} \left| \bs \right|_{\infty},
$$
where $C$ does not depend on $\e, \eta$ but may depend on $A$. 
\end{Lemma}
\begin{proof} By  definition we have
\begin{align*} 
\mathcal{E}_{tail} &= \dfrac{1}{\tau} \dfrac{1}{\eta^{d}} \int_{-\tau}^{\tau} \int_{\Omega_{\eta}} K(t/\tau) K(\bx/\eta) A(\bx,\bx/\e) \left(  \nabla u^{\e,\eta}  - \nabla \tilde{u}^{\e,\eta} \right) \; d\bx \; dt  \\
&\leq \dfrac{1}{\tau} \dfrac{1}{\eta^{d}} \int_{-\tau}^{\tau}  \left| K(t/\tau) \right| \left| \int_{\Omega_{\eta}} K(\bx/\eta) A(\bx,\bx/\e) \left(  \nabla u^{\e,\eta}  - \nabla \tilde{u}^{\e,\eta} \right) \; d\bx \right| \; dt \\
&\leq  C \dfrac{1}{\tau} \dfrac{1}{\eta^{d}} \int_{-\tau}^{\tau}  \left| K(t/\tau) \right|   \sup_{t \in [0,\tau]} \| \nabla u^{\e,\eta} - \nabla \tilde{u}^{\e,\eta} \|_{L^2(\Omega_{\eta})} \eta^{d/2} \; dt \\
&\leq C \eta^{-d/2} \sup_{t\in [0,\tau]} \| \nabla u^{\e,\eta} - \nabla \tilde{u}^{\e,\eta} \|_{L^2(\Omega_{\eta})} \leq C   \e^{2} \left( \dfrac{\eta}{\e} \right)^{3} \left( \dfrac{\tau}{\e} \right)^{4} \left| \bs \right|_{\infty} \leq \e^{-5} \eta^{7} \left| \bs \right|_{\infty}.
\end{align*}
In the last two inequalities we used Corollary \ref{Lemma_Tail_Estimate} and the fact that $\tau=  \eta$, respectively.
\end{proof}
Before proving Theorem \ref{Thm_Main_Thm}, we  introduce some intermediate results.
\begin{Lemma}\label{Lemma_d00Chi_Estimate}
Suppose that $d_{00}$ is given by \eqref{Eqn_TimeAverages_d} with  $\br_0=\bx$. Moreoever let the assumptions of Theorem \ref{Thm_TimeAverages_d} be satisfied and
$$
z_{00}(\bx,\by) = \sum_{\ell=1}^{d} s_{\ell} \chi_{\ell}(\bx,\by),
$$
where $\chi_{\ell}$ is the solution of the cell problems \eqref{eqn_CellProblemChi}. Then
\begin{align*}
\left\| d_{00}(\bx,\cdot)  - z_{00}(\bx,\cdot) \right\|_{H^{1}(Y)} \leq C \alpha^{q} \left| \bs \right|_{\infty},
\end{align*} 
where $\alpha = \e/\eta$ and $C$ does not depend on $\alpha, \e ,\eta$ but may depend on $A,p,q$ or $K$.
\end{Lemma}
\begin{proof} From \eqref{eqn_CellProblemChi}, we see that
\begin{align} \label{Eqn_TimeAverages_z}
L[z_{00}] &= -\nabla_y \cdot A \bs 
\end{align}
Note the similarity between \eqref{Eqn_TimeAverages_z} with the first equation in \eqref{Eqn_TimeAverages_d}. Using these two equations we readily see that
\begin{equation*}
L[d_{00}  - \sum_{\ell=1}^{d} s_{\ell} \chi_{\ell}] = \alpha^{q} Z_{00}(\bx,\by).
\end{equation*}
Therefore by elliptic regularity, Lemma \ref{Lemma_Elliptic_Regularity}, and the estimate \eqref{Eqn_Z_Estimates} we have
\begin{align*}
\left\| d_{00}(\bx,\cdot)  - z_{00}(\bx,\cdot) \right\|_{H^2(Y)} \leq \alpha^{q} \left\| Z_{00}(\bx,\cdot) \right\|_{L^2(Y)} \leq C |\bs|_{\infty} \alpha^{q}.
\end{align*}
\end{proof}
Now we present a lemma from  \cite{Arjmand_Runborg_1} which concerns the local averages of locally-periodic functions.
\begin{Lemma} \label{Lemma_Kernel}
Let $f$ be a $1$-periodic continuous function and $K\in \mathbb{K}^{p,q}$. Then, with $\alpha = \varepsilon/\eta \leq 1$, and $\bar{f} = \int_{0}^{1} f(s) ds$ 
\begin{equation*}
\left| \int_{\mathbb{R}} K_{\eta}(t) f(t/\varepsilon) dt - \bar{f} \right| \leq C |f|_{\infty} \alpha^{q+2}, 
\end{equation*}
\noindent and when $r \in \mathbb{Z}^{+}$ and $\eta <1$,
\begin{equation*}
\left| \int K_{\eta}(t) t^r f(t/\varepsilon) dt \right| \leq C 
\begin{cases}
|f|_{\infty} \alpha^{q+2} \eta^r & 1\leq r \leq p \\
|f|_{\infty}  \alpha^{q+2} \eta^r  + |\bar{f}| \eta^r & r >p,
\end{cases}
\end{equation*}
\noindent where the constant $C$ does not depend on $\varepsilon$, $\eta$, $f$ or $s$ but may depend on $K,p,q,r$.
\end{Lemma}
We are now ready to present the proof of Theorem \ref{Thm_Main_Thm}.
\subsection{Proof of Theorem \ref{Thm_Main_Thm}} \label{SubSec_MainProof} Without loss of generality, we consider $\br_0=0$, {\color{blue} and to simplify the notation we write $\alpha=  \e/\eta$}. Moreover, all the upper bounds will be uniform in $\br_0$ as we have assumed that $\br_0$ belongs to a compact set. We use the decomposition of $\bF$ in the form \eqref{Eqn_HMM_Flux_Decomposed}, and split the error as follows:
\begin{align*}
|  \bF - A^{0}\bs  |_\infty \leq \underbrace{|\bF_{0}  - A^{0} \bs |_{\infty}}_{I} + \e \underbrace{| \bF_{1}|_{\infty}}_{II}  + | \delta |_{\infty} + | \mathcal{E}_{tail} |_{\infty}.  
\end{align*}
The estimate $| \mathcal{E}_{tail} |_{\infty} \leq C |\bs|_{\infty} \e^{-5} \eta^{7}$ holds by Lemma \ref{Lemma_Tail_Flux_Estimate}. Moreover, by Lemma \ref{Lemma_Kernel}, and the assumption that $p>1$, it follows that $| \delta |_{\infty} \leq C|\bs|_{\infty} \eta \alpha^{q+2} $. It remains to bound the first two terms. To bound the first term we first introduce 
$$
\tilde{F}_{0,j} :=  e_j^T \int_{Y} A(0,\by) \left( \bs + \nabla_\by d_{00}(0,\by)  \right) \;d\by.
$$
Then $\displaystyle I:= \max_{j=1,\ldots,d} I_{j}$, where
$$
I_j := \left|e_j^T \left( \bF_{0}  - \hat{A}(0) \bs \right) \right| \leq \left| e_j^T \bF_{0} - \tilde{F}_{0,j} \right| + \left| \left(  \tilde{F}_{0,j} -  e_j^T A^{0} \bs \right) \right|.
$$
By Lemma \ref{Lemma_Kernel} we have
$
\left| e_j^T \bF_{0} -  \tilde{F}_{0,j} \right| \leq C |\bs|_{\infty}  \alpha^{q+2}.
$
Moreover, by \eqref{eqn_C_HatA_Def}
$$
e_j^{T} A^{0}(0) \bs = e_j^{T} \int_{Y} A(0,\by) \left( \bs  + \nabla_\by \sum_{\ell=1}^{d} s_{\ell} \chi_{\ell} \right) \; d\by.
$$
Therefore,
\begin{align*}
\left| \tilde{F}_{0,j} - e_{j}^{T} A^{0} \bs \right|  &= \left| e_{j}^T \int_{Y} A(0,\by) \nabla_\by \left( d_{00} - \sum_{\ell=1}^{d} s_{\ell} \chi_{\ell}(0,y)  \right) \; d\by \right| \\
&\leq C \| d_{00}(0,\cdot) -  \sum_{\ell=1}^{d} s_{\ell} \chi_{\ell}(0,\cdot) \|_{H^1(Y)} \leq C \alpha^{q},
\end{align*}
where we used Lemma \eqref{Lemma_d00Chi_Estimate} in the last inequality. This proves that $I \leq C \alpha^{q} |\bs|_{\infty}$.  To bound the term $II$, we first introduce
$$
\tilde{F}_{1,j}=  e_j^T \int_{Y} A(0,\by) \left( \nabla_\by d_{10}(0,\by)  + d_{11}(0,\by)  \right) \; d\by.
$$
Then we can write $\displaystyle II:= \max_{j=1,\ldots,d} II_j $, where
$$
II_j = |e_j^{T}   \bF_{1} - \tilde{F}_{1,j}|_{\infty} + |  \tilde{F}_{1,j} |_{\infty}
$$
For the first term, we use Lemma \ref{Lemma_Kernel} to obtain $|e_j^{T}   \bF_{1} - \tilde{F}_{1,j}|_{\infty} \leq C |\bs|_{\infty} \alpha^{q+2}$. For the second term, however, we apply elliptic regularity, Lemma \ref{Lemma_Elliptic_Regularity}, to equations in \eqref{Eqn_TimeAverages_d} and see that 
$$
 |  \tilde{F}_{1,j} |_{\infty} \leq C  |\bs|_{\infty}.
$$
 To complete the proof we now show that $|\tilde{F}_{1,j}| \leq C \alpha^{q}$ when $d=1$. For this we take the $y$ derivative of the last equation in \eqref{Eqn_TimeAverages_d} and use the equation for $d_{11}:=d_{11j}$ (in one dimension  we have $j=1$) and obtain:
\begin{align*}
\partial_y^{2} \left( A \left( \partial_y d_{10}  + d_{11}  \right) \right) &= -\partial_y \left(  A \partial_y d_{11}  \right) - \partial_{y} \partial_x A s - \partial_y \left(  \left( \partial_x A \right) \partial_y d_{00} \right) + \alpha^{q} \partial_y Z_{00}(y) \\
&= \alpha^{q}  \left(  \partial_y Z_{00}- Z_{111}\right).
\end{align*}
Let $\chi$ be the solution of the cell-problem \eqref{eqn_CellProblemChi} in one dimension then $\partial_y \chi = -1 + A^{0} A^{-1}(0,y)$. From here, for a smooth and $1$-periodic function $u$  we easily obtain
$$\langle \chi  , \partial_y  \left( A u \right)\rangle  = \int_Y A(0,y) u(y) \; dy  - A^{0} \bar{u}.$$
Next we define $h(z):= \int_{0}^{z} \chi(y) \; dy$, then $h$ is $1$-periodic since $\bar{\chi} =0$. We use the above relation for $u = \partial_y d_{00} + d_{11}$ and exploit the fact $\overline{\partial_y d_{00} + d_{11}}=0$ to see that 
\begin{align*}
\left| F_{1,j} \right| & := \left| \int_{Y} A(0,y) \left(  \partial_y d_{00} + d_{11} \right) \; dy \right|  = \left| \langle  \chi, \partial_y \left(  A \left(\partial_y d_{00} + d_{11} \right) \right) \rangle \right| \\ 
& =  \left| \langle h , \partial_y^{2} \left(  A \left(\partial_y d_{00} + d_{11} \right) \right) \rangle \right| = \left| \langle h , \alpha^{q}  \left(  \partial_y Z_{00}- Z_{111}\right) \rangle \right| \\ &\leq \alpha^{q}\| h \|_{L^2(Y)} \left( \| Z_{00} \|_{H^1(Y)}  + \|Z_{111} \|_{H^1(Y)}   \right) \leq C \alpha^{q-1} |s|.
\end{align*}
Note that in the last inequality we used the estimates \eqref{Eqn_Z_Estimates}. The proof of the Theorem \ref{Thm_Main_Thm} is completed. 
\section{Conclusion}
\label{Discussion_Sec}
In this paper, we have proved convergence rates for the upscaling error in a FD-HMM, for wave propagation problems in locally periodic media. The analysis extends the results from the periodic theory, and reveals the precise convergence rates for the difference between the exact/homogenized and numerically upscaled fluxes. The outcomes of the present work are: a) in locally periodic media, in addition to the errors predicted by the periodic theory,  another error appears due to the interaction of the slow and fast scales in the media. These errors are precisely quantified in the present analysis. b) In general, the upscaling error in HMM type algorithms may result in different asymptotic convergence rates depending on the dimension, where an improved convergence rate is typically observed in one dimension, see Figure \ref{Fig_FluxConv1D2D} for numerical results in one and two dimensions. The results in the current study give a complete theoretical explanation of this dimension-dependent phenomenon.

\section{Appendix}
Our aim in this section is to prove Theorem \ref{Truncated_Convergence_Thm}. We first introduce few intermediate results.  The first lemma is from \cite{Arjmand}.
\begin{Lemma}\label{Bounded_Region_Wave_Stability_Lemma}Suppose that $A \in (C^{\infty}(\mathbb{R}))^{d\times d}$,  $f\in L^1 \left( 0,T;H^1_{loc}\left( \mathbb{R}^d \right) \right)$, and 
\begin{equation*}
\begin{array}{ll}
\partial_{tt} u(t,\bx)  - \nabla \cdot \left( A(\bx) \nabla u(t,\bx) \right) = f(t,\bx) \text{   in } \mathbb{R}^d \times (0,T],\\
u(0,\bx) = 0, \quad \partial_t u(0,\bx) = 0 \text{ on } \mathbb{R}^d \times \{ t = 0\}. 
\end{array}
\end{equation*}
Let 
$$
E_{u,\Omega}(t) =  \int_{\Omega} |\partial_t u|^2  + A \nabla u \cdot \nabla u \; d\bx,
$$
and $M > L + t \sqrt{\parallel A \parallel_{\infty}}$. Moreove, let $\Omega_{L} = [-L,L]^d$ with an obvious change for $M$, then the solution $u$ satisfies
\begin{equation}\label{Bounded_Region_Wave_Stability_Lemma_Eqn_1}
 E_{u,\Omega_L}^{1/2}(t) \leq C  \int_{0}^{t} \parallel f(s,\cdot) \parallel_{L^2\left( \Omega_{M}\right)} ds,
\end{equation}
where $C$  does not depend on $T$, but may depend on $A$ and $d$.
\end{Lemma}
We consider next the following utility lemma:
\begin{Lemma} \label{Lemma_fg}Let $\Omega_{M}=[-M,M]^{d}$ and $Y=[0,1]^{d}$. Moreover, suppose that 
$g=x_j f$, where $f\in H^{k}(Y)$ is a $Y$-periodic function. Then
\begin{align*}
\|  f \|_{H^k(\Omega_{M})} \leq C \left( 1 + M^{d/2} \right) \| f \|_{H^k(Y)},
\end{align*}
and 
\begin{align*}
\|  g \|_{H^k(\Omega_{M})} \leq C \left( 1 + M^{d/2+1} \right) \| f \|_{H^k(Y)},
\end{align*}
where $C$ does not depend on $f,g,$ or $M$ but may depend on $d$. 
\end{Lemma}
\begin{proof} 
First we observe that
\begin{align*}
\| f \|^{2}_{H^k(\Omega_{M})} &= \int_{\Omega_{M}} |f|^{2}  + |\nabla f|^{2} + \ldots  + |\nabla^{k} f|^{2} \; d\bx \\
&\leq \left( 1 + M^{d} \right) \int_{Y}  |f|^{2}  + |\nabla f|^{2} + \ldots  + |\nabla^{k} f|^{2} dx  = \left( 1 + M^{d} \right) \| f \|_{H^k(Y)}^{2}.
\end{align*}
This proves the first part of the lemma. For $g$ we have
\begin{align*}
\| g \|^{2}_{H^k(\Omega_{M})} &= \int_{\Omega_{M}} |g|^{2}  + |\nabla g|^{2} + \ldots  + |\nabla^{k} g|^{2} \; d\bx \\
&\leq  \int_{\Omega_{M}} |x_j f|^2 + |f|^2 + |x_j f|^2  + \ldots  + |f|^2  + |x_j \nabla^{k} f|^2 \; d\bx \\
&\leq \left( 1 + M^{2} \right)  \| f \|_{H^k(\Omega_{M})}^{2} \leq \left( 1 + M^{d+2} \right) \| f \|_{H^k(Y)}^{2}.
\end{align*}
\end{proof}

Now we use Lemma \ref{Lemma_fg} to estimate the time growth of $v_0$ and $v_1$ solving \eqref{Eq_v0} and \eqref{Eq_vm} respectively.

\begin{Lemma} \label{Lemma_TimeGrowth_v0_v1} Let $v_0$ and $v_1$ be the solutions of the wave equations \eqref{Eq_v0} and \eqref{Eq_vm} respectively. Moreover let $\Omega_{M}:=[-M,M]^d$, then
\begin{align*}
\max_{|z|\leq t}\| v_0(z, \cdot) \|_{H^2(\Omega_M)} &\leq C \left( 1 + M^{d/2 + 1} \right) \left| \bs \right|_{\infty}, 
\\
\max_{|z|\leq t} \quad \| v_1(z, \cdot) \|_{H^2(\Omega_M)} &\leq C \left( 1  + M^{d/2 + 1} \right)  \left( 1+ t^{3} \right)\left| \bs \right|_{\infty}.
\end{align*}
where $C$ does not depend on $M,t$ but may depend on $A$ and $d$.
\end{Lemma}
\begin{proof} Using the relation $v_0 = v_{00} + \bs \cdot \bx$, toghether with Lemma \ref{Lemma_fg} and the estimate \eqref{Eqn_Energy_Estimates_v} we get
\begin{align*}
\| v_0(t,\cdot) \|_{H^2(\Omega_{M})} &\leq \| v_{00}(t,\cdot) \|_{H^2(\Omega_{M})} + \sum_{j=1}^{d} \| s_j x_j  \|_{H^2(\Omega_{M})} \\
&\leq  C_1 \left(  1 + M^{d/2} \right) \| v_{00}(t,\cdot) \|_{H^2(Y)}  + C_2 \left( 1 + M^{d/2 + 1} \right)
\\ &\leq C \left( 1 + M^{d/2 + 1} \right) \left| \bs \right|_{\infty}.
\end{align*}
To prove the estimate for $v_1$, we first recall from \eqref{Eqn_g_Equation} that
\begin{align*}
g^{\prime\prime}(t) &=  \overline{M[v_{11}]} + \overline{f_{10}}, \\
g(0)  &= g^{\prime}(0) = 0.
\end{align*}
We then use the estimate for $v_{11j}$ in \eqref{Eqn_Energy_Estimates_v} and the inequality \eqref{Eqn_f10_Estimate} and get
\begin{align*}
\left| g^{\prime\prime}(t) \right| &\leq C_1 \left|  \overline{M[v_{11}]}(t) \right|  + C_2 \left| \overline{f_{10}(t,\cdot)} \right| \\
&\leq C_1 \|  M[v_{11}](t,\cdot)\|_{L^{2}(Y)}  + C_2 \|f_{10}(t,\cdot) \|_{L^2(Y)} \\
&\leq C_1 \sum_{j=1}^{d} \| v_{11j}(t,\cdot) \|_{H^1(Y)} + C_2 |\bs|_{\infty}  \leq C \left( 1+ t \right) |\bs|_{\infty}. 
\end{align*}
Using the conditions $g(0) = g^{\prime}(0) =0$, we obtain
\begin{align} \label{Eqn_Estimate_g}
\left| g(t) \right| \leq \int_{0}^{t} \int_{0}^{s} \left| g^{\prime\prime}(z) \right| dz ds \leq C \left( 1 + t^{3} \right).
\end{align}
Then by Lemma \ref{Lemma_fg} and the estimates \eqref{Eqn_Energy_Estimates_v} and \eqref{Eqn_Estimate_g} it follows that
\begin{align*}
\| & v_1(t,\cdot) \|_{H^2(\Omega_{M})} \leq \|  v_{10}(t,\cdot) \|_{H^2(\Omega_{M})} + \sum_{j=1}^{d} \| x_j v_{11j}\|_{H^{2}(\Omega_{M})} \\
&\leq \left( 1 + M^{d/2} \right) \|  v_{10}(t,\cdot) \|_{H^2(Y)}  + \left( 1  + M^{d/2 + 1} \right) \sum_{j=1}^{d}  \| v_{11j}(t,\cdot)\|_{H^{2}(Y)} \\ 
&\leq \left( 1 + M^{d/2} \right) \left( \|  \tilde{v}_{10}(t,\cdot) \|_{H^2(Y)}   + \left|g(t) \right|\right) + \left( 1  + M^{d/2 + 1} \right) \sum_{j=1}^{d}  \| v_{11j}(t,\cdot)\|_{H^{2}(Y)} \\
&\leq \left( 1 + M^{d/2} \right) \left(1 + t^{3} \right) |\bs|_{\infty} + \left( 1  + M^{d/2 + 1} \right) \left( 1+ t\right) |\bs|_{\infty} \\ &\leq \left( 1  + M^{d/2 + 1} \right)  \left( 1+ t^{3} \right)  |\bs|_{\infty}.
\end{align*}
\end{proof}
Now we give the proof of Theorem \ref{Truncated_Convergence_Thm}.
\begin{proof}(of Theorem \ref{Truncated_Convergence_Thm}) The proof of this theorem is based on deriving an explicit equation for the error $e(t,x) = v(t,x) - \tilde{v}_m(t,x)$ and estimating the error using Lemma \ref{Bounded_Region_Wave_Stability_Lemma}. We start with introducing the truncated Taylor expansion of $A(x,y)$ (denoted by $\tilde{A}_n$) in terms of $x$ in multi-dimensions. We use a  multi-index notation (with index $\beta$) as follows:   

\begin{equation}\label{Taylor_Expansion_A}
\tilde{A}_n(\bx,\by)  = \sum_{|\beta| \leq n} \frac{\bx^{\beta}}{\beta !} \partial_\bx^{\beta} A(0,y). 
\end{equation}

In addition, we denote the tail of the expansion by $\delta A_n(\bx,\by)  = A(\bx,\by) - \tilde{A}_n(\bx,\by)$. Taylor's formula in multi-dimensions gives then 
\begin{equation} \label{Estimate_Tail_A}
\begin{array}{ll}
\left|  \delta A_n(\varepsilon \bx,\bx) \right|  &=  \left| A(\varepsilon \bx, \bx) - \tilde{A}_n(\varepsilon \bx,\bx) \right| \\
 &\hspace{-0.05cm}= \displaystyle \left| \sum_{|\beta| =n+1} \frac{n+1}{\beta !} \varepsilon^{|\beta|} \bx^{|\beta|} \int_{0}^1 (1-t)^n \left( \partial^{\beta}_\bx A(\bx,\by)|_{\bx = \varepsilon \bx t}\right)|_{\by=\bx} dt   \right| \\
&\leq C_n \varepsilon^{n+1} \left| \bx  \right|^{n+1} \displaystyle \max_{|\beta| = n+1} | \partial_\bx^{\beta} A |, 
\end{array}
\end{equation}
and
\begin{equation} \label{Estimate_Tail_Partial_A}
\begin{array}{ll}
\left| \partial_{x_{k}} \delta A_n(\varepsilon \bx,\bx) \right|  &=  \left| \partial_{x_{k}} \left( A(\varepsilon \bx, \bx) - \tilde{A}_n(\varepsilon \bx,\bx)\right) \right| \\
& = \displaystyle \left| \partial_{x_k} \sum_{|\beta| =n+1} \frac{n+1}{\beta !} \varepsilon^{|\beta|} \bx^{|\beta|} \int_{0}^1 (1-t)^n  \partial^{\beta}_\bx A(\bx,\by)|_{\bx = \varepsilon \bx t} \; dt |_{\by=\bx} \right| \\
&\leq C_n \varepsilon^{n+1} \left| \bx  \right|^{n} \left( 1  + \left| \bx \right| \right) \displaystyle \max_{|\beta| = n+2}| \partial_\bx^{\beta} \partial_\by A |. 
\end{array}
\end{equation}

Next by equation \eqref{Eq_vm} and the definition of $\tilde{v}_m$ we have
\begin{equation*}
\begin{array}{lll}
\partial_{tt} \tilde{v}_m(t,\bx)  &=& \sum_{k=0}^{m} \frac{\varepsilon^k}{k!} \partial_{tt} v_k(t,\bx) \\
&=& \sum_{k=0}^{m} \frac{\varepsilon^k}{k!} \sum_{j=0}^{k} \nabla \cdot \left( c_{kj} \left( \partial_{\varepsilon}^{k-j}  A(\varepsilon \bx, \bx)\big|_{\varepsilon=0} \right) \nabla v_j  \right)  \\ 
& =& \nabla \cdot \left( \sum_{j=0}^{m} \sum_{k=j}^{m} \frac{\varepsilon^k}{j! (k-j)!} \left( \partial_{\varepsilon}^{k-j}  A(\varepsilon \bx, \bx)\big|_{\varepsilon=0} \right) \nabla v_j \right)\\
& =& \nabla \cdot \left( \sum_{j=0}^{m} \frac{\varepsilon^j}{j!} \sum_{k=0}^{m-j} \frac{\varepsilon^k}{k!} \left( \partial_{\varepsilon}^{k}  A(\varepsilon \bx, \bx)\big|_{\varepsilon=0} \right) \nabla v_j \right) \\
& = &\nabla \cdot \left( \sum_{j=0}^{m} \frac{\varepsilon^j}{j!} \sum_{k=0}^{m-j} \frac{\varepsilon^k}{k!}  \sum_{|\beta| = k} \frac{|\beta|!}{\beta!} x^{\beta}  \partial_{\bx}^{\beta}  A(0, \by)\big|_{\by=\bx} \nabla v_j \right) \\
&=&\nabla \cdot \left( \sum_{j=0}^{m} \frac{\varepsilon^j}{j!} \sum_{|\beta| \leq m-j} \frac{\varepsilon^{|\beta|}}{\beta!} \bx^{\beta}   \partial_{\bx}^{\beta}  A(0, \by)\big|_{\by=\bx} \nabla v_j \right) \\
&=& \nabla \cdot \left( \sum_{j=0}^{m} \frac{\varepsilon^j}{j!} \tilde{A}_{m-j}(\varepsilon \bx,\bx) \nabla v_j \right).
\end{array}
\end{equation*}

Furthermore, by definition of $e(t,\bx)$ and equation \eqref{eqn_ScaledMainProblem_2D} we obtain
\begin{equation*}
\begin{array}{llll}
\partial_{tt} e &= \partial_{tt} v  - \partial_{tt} \tilde{v}_m \\ 
&= \nabla \cdot \left( A(\varepsilon \bx,\bx ) \nabla v \right) - \nabla \cdot \left( \sum_{j=0}^{m} \frac{\varepsilon^j}{j!} \tilde{A}_{m-j}(\varepsilon \bx,\bx) \nabla v_j \right) \\
&= \nabla \cdot \left(A(\varepsilon \bx,\bx ) \nabla e  \right)  + \nabla \cdot \left( \sum_{j=0}^{m} \frac{\varepsilon^j}{j!} \left( A(\varepsilon \bx,\bx) - \tilde{A}_{m-j}(\varepsilon \bx,\bx) \right) \nabla v_j \right) \\
&= \nabla \cdot \left(A(\varepsilon \bx,\bx ) \nabla e  \right) + f(t,\bx).
\end{array}
\end{equation*}

Together with zero initial data, $e$ satisfies then
\begin{equation*}
\begin{array}{ll}
\partial_{tt} e(t,\bx)  - \nabla \cdot \left( A(\varepsilon \bx, \bx) e(t,\bx) \right) = f(t,\bx), \\ e(0,\bx) = \partial_t e(0,\bx) = 0,
\end{array}
\end{equation*}
where,  
\begin{equation*}
f(t,\bx) = \nabla \cdot \left( \sum_{j=0}^{m} \frac{\varepsilon^j}{j!} \left( A(\varepsilon \bx,\bx) - \tilde{A}_{m-j}(\varepsilon \bx,\bx) \right) \nabla v_j \right).
\end{equation*}

By Lemma \ref{Bounded_Region_Wave_Stability_Lemma}  we have $E_{e,\Omega_L}^{1/2}(t) \leq C \int_{0}^{t}\parallel f(s,\cdot) \parallel_{L^2 \left( \Omega_{M} \right)} \; ds$. But
\begin{equation*}
\begin{array}{lll}
& \displaystyle \int_{0}^{t} \parallel f(s,\cdot) \parallel_{L^2\left( \Omega_{M} \right)} \; ds  = \displaystyle \int_{0}^{t} \parallel \nabla \cdot \left( \sum_{j=0}^{m} \frac{\varepsilon^j}{j!} \left( \delta A_{m-j}(\varepsilon  \cdot,\cdot) \right) \nabla v_j(s,\cdot) \right) \parallel_{L^2\left( \Omega_{M} \right)}  ds \\
& = \displaystyle \int_{0}^{t} \parallel \sum_{j=0}^{m} \frac{\varepsilon^j}{j!} \left( \nabla \cdot  \delta A_{m-j}(\varepsilon \cdot,\cdot) \cdot \nabla v_j(s,\cdot)   +   \delta A_{m-j}(\varepsilon \cdot, \cdot):\nabla^2 v_j(s,\cdot)  \right) \parallel_{L^2\left( \Omega_{M} \right)}  ds \\
& \leq \sum_{j=0}^{m} \frac{\varepsilon^j}{j!} C_d     \displaystyle \max_{x\in \Omega_M, 0\leq |\beta| \leq 1 }|\partial_x^{\beta} \delta A_{m-j}(\varepsilon \bx, \bx)| \int_{0}^{t} \| v_j(s,\cdot) \|_{H^2(\Omega_M)} \; ds.
\end{array}
\end{equation*}
Next by \eqref{Estimate_Tail_A} and \eqref{Estimate_Tail_Partial_A} we have $$ \displaystyle \max_{\bx\in \Omega_M, 0\leq |\beta| \leq 1 }|\partial_\bx^{\beta} \delta A_{m-j}(\varepsilon \bx, \bx)| \leq C_{A,m} \varepsilon^{m-j+1} M^{m-j} \left( 1 + M \right).$$
Consequently,
\begin{equation*}
\begin{array}{ll}
E_{e,\Omega_L}^{1/2}(t) & \leq C_{A,m,d}  \varepsilon^{m+1} \sum_{j=0}^{m} M^{m-j} \left(1 + M \right) t \max_{|z|\leq t} \left\| v_{j}(z,\cdot) \right\|_{H^2(\Omega_M)}.
\end{array}
\end{equation*}
The first part of the theorem follows then by ellipticity of $A$. To prove the second part of the theorem we put $m=1$ in the last inequality and use Lemma \ref{Lemma_TimeGrowth_v0_v1} and obtain
\begin{align*}
& \| \nabla e \|_{L^{2}(\Omega_{M})} \leq \\ &\leq C \e^2 \left(   M \left( 1+ M \right)  t \max_{|z|\leq t} \| v_0(z,\cdot) \|_{H^2(\Omega_{M})} + \left(1 + M \right) t \max_{|z|\leq t} \| v_1(z,\cdot) \|_{H^2(\Omega_{M})}\right) \\
&\leq C \e^{2}\left( 1+ M^{d/2 + 1} \right) \left( 1 +M \right) t |\bs|_{\infty} \left( M    +  \left( 1+t^3 \right)\right) \\ &\leq  C \e^{2}\left( 1+ M^{d/2 + 1} \right) \left( 1 +M^2 \right) \left( 1 + t^4 \right) |\bs|_{\infty}.
\end{align*}
\end{proof}


\bibliographystyle{spmpsci}
\bibliography{references}

\end{document}